\documentclass[10pt]{article}
\usepackage[cp1251]{inputenc}
\usepackage{amsfonts}
\usepackage{amsmath}
\usepackage{mathrsfs}
\usepackage{color}
\usepackage{tikz}
\usepackage{mathrsfs,amscd,amssymb,amsthm,amsmath,bm,graphicx,psfrag,subfigure,url}
\usepackage[titletoc]{appendix}
\usepackage{epstopdf}

\usepackage{indentfirst}

\def \[{\begin{equation}}
\def \]{\end{equation}}

\newtheorem{thm}{Theorem}[section]
\newtheorem{prop}[thm]{Proposition}

\newtheorem{fact}[thm]{Fact}
\newtheorem{lem}[thm]{Lemma}

\begin{document}
\begin{center}{\Large \bf Spectral radius and the $2$-power of Hamilton cycles}
\vspace{4mm}

{\large Xinru Yan, \ \ Xiaocong He, \ \ Lihua Feng\footnote{Corresponding author.  \\ \hspace*{5mm}{\it Email addresses}: 2917251789@qq.com (X.R. Yan)\ \ fenglh@163.com (L.H. Feng)}, \ \ Weijun Liu.}
\vspace{3mm}

School of Mathematics and Statistics, HNP-LAMA,  Central South University, New Campus, \\
Changsha, Hunan, 410083, PR China\\[5pt]
\end{center}


\noindent {\bf Abstract}: \
Let $G$ be a graph of order $n$ and spectral radius be the largest eigenvalue of its adjacency matrix, denoted by $\mu(G)$. In this paper, we determine the unique graph with maximum spectral radius among all graphs of order $n$ without containing the $2$-power of a Hamilton cycle.

\vspace{2mm} \noindent{\bf Keywords}: $2$-power of graphs; Hamilton cycle; Spectral radius.

\vspace{2mm} \noindent{\bf AMS classification}: 05C50, 05C35
\vspace{2mm}

\setcounter{section}{0}
\section{\normalsize Introduction}\setcounter{equation}{0}
We will start with introducing some background information that will lead to our main results. Some important previously established facts will also be presented.
\subsection{\normalsize Background}
Let $G=(V(G),E(G))$ be a simple graph and $\mu(G)$ be the largest eigenvalue of its adjacency matrix, where $V(G)$ and $E(G)$ are vertex set and edge set, respectively. For $u, v \in V(G)$, the distance between $u$ and $v$ is defined to be the number of edges of a shortest path from $u$ to $v$. We write $e(G)$ for $|E(G)|$, $d_G(u)$ stands for the degree of the vertex $u$ in $G$, and $\delta(G)$ stands for the minimum degree in $G$. We use $K_n$, $S_n$, $P_n$ and $C_n$ to denote the complete graph, the star, the path and the cycle of order $n$, respectively. Note that $C_n$ is also called a Hamilton cycle of $G$ with order $n$. The join of two disjoint graphs $G_1$ and $G_2$, denoted by $G_1 \vee G_2$, is the graph obtained from $G_1 \cup G_2$ by joining each vertex of $G_1$ to each vertex of $G_2$. For odd number $n$, we use $F_n$ to stand for a union of $\lfloor {\frac{n}{2}} \rfloor$ triangles sharing a single common vertex. $K_{a, b}$ denotes the complete bipartite graph with vertex partition sets of sizes $a$ and $b$. Let $S_{n, k}$ be the graph obtained by joining each vertex of $K_{k}$ to $n-k$ isolated vertices. A wheel $W_{n}$ is the graph obtained by joining $C_{n-1}$ with an additional vertex. Let $G+e$ denote the graph obtained from $G$ by adding an edge $e$ between a pair of non-adjacent vertices. Let $G^{-}$ denote the set of graphs obtained from $G$ by deleting any edge and $G^{+}$ denote the set of graphs obtained from $G$ by adding a new vertex and joining it to any one vertex of $G$. If there is only one non-isomorphic graph in $G^{-}$ or $G^{+}$, then we also use $G^{-}$ or $G^{+}$ to denote this unique graph. Also, $G^{+t}$ denotes the set of graphs obtained from $G$ by adding a new vertex and joining it to any $t$ vertices of $G$. Let $\overline{G}$ denote the complement graph of $G$, where any two vertices in $\overline{G}$ are adjacent if they are not adjacent in $G$. The $k$-power of a graph $G$, denoted by $G^{k}$, is another graph that has the same set of vertices, but in which two vertices are adjacent when their distance in $G$ is at most $k$. Let $H$ be any subgraph of $G$. Then $G \backslash E(H)$ denotes the graph obtained from $G$ by deleting edges of $H$. We adopt the notation and terminologies in \cite{bondy, 19, 6, 18, 20} except as stated otherwise.

Let $\mathcal{K}(n, t)$ be the set of graphs on $n$ vertices with $t$ edges. Let $M_n$ be the $2n$-vertex graph on $n$ independent edges. Let $S_{n}^{*}$ be the graph obtained from $S_{n}$ by adding a new vertex and joining it to a leaf of $S_{n}$. Let $T_{a, b, c}$ stand for a $T$-shaped tree defined as a tree with a single vertex $u$ of degree $3$ such that $T_{a, b, c}-u=P_{a} \cup P_{b} \cup P_{c}~(a \leqslant b \leqslant c)$. Similarly, $T_{a, b, c, d}$ stands for a $T$-shaped tree defined as a tree with a single vertex $u$ of degree $4$ such that $T_{a, b, c, d}-u=P_{a} \cup P_{b} \cup P_{c} \cup P_{d}~(a \leqslant b \leqslant c \leqslant d)$ and $T_{a, b, c, d, e}$ can be defined in the same way. Let $D_n$ be the double-snake of order $n$, depicted in Figure 1. Let $\mathcal{L}_{n}$ be a set of graphs with $n \geqslant 9$ and the graphs in $\mathcal{L}_{n}$ are depicted in Figure 2.
\begin{figure}[h!]\label{fig1}
\begin{center}
\psfrag{1}{$\cdots$}\psfrag{2}{$D_n$}
\includegraphics[width=60mm]{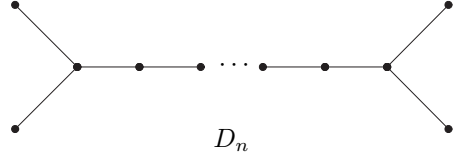} \\
  \caption{The Double-snake.}
\end{center}
\end{figure}

\begin{figure}[h!]\label{fig2}
\begin{center}
\psfrag{1}{$\cdots$}\psfrag{2}{$n-2$}\psfrag{3}{$n-3$}\psfrag{4}{$n-4$}\psfrag{5}{$n-5$}\psfrag{6}{$n-6$}\psfrag{7}{$n-7$}\psfrag{8}{$n-8$}
\psfrag{a}{$S_{n-4}\bigcup S_4$}\psfrag{b}{$S_{n-4}\bigcup K_3$}\psfrag{c}{$S_{n-1}$}\psfrag{d}{$S^*_{n-2}$}\psfrag{e}{$S_{n-2}+e$}\psfrag{f}{$S_{n-2}\bigcup K_2$}
\includegraphics[width=160mm]{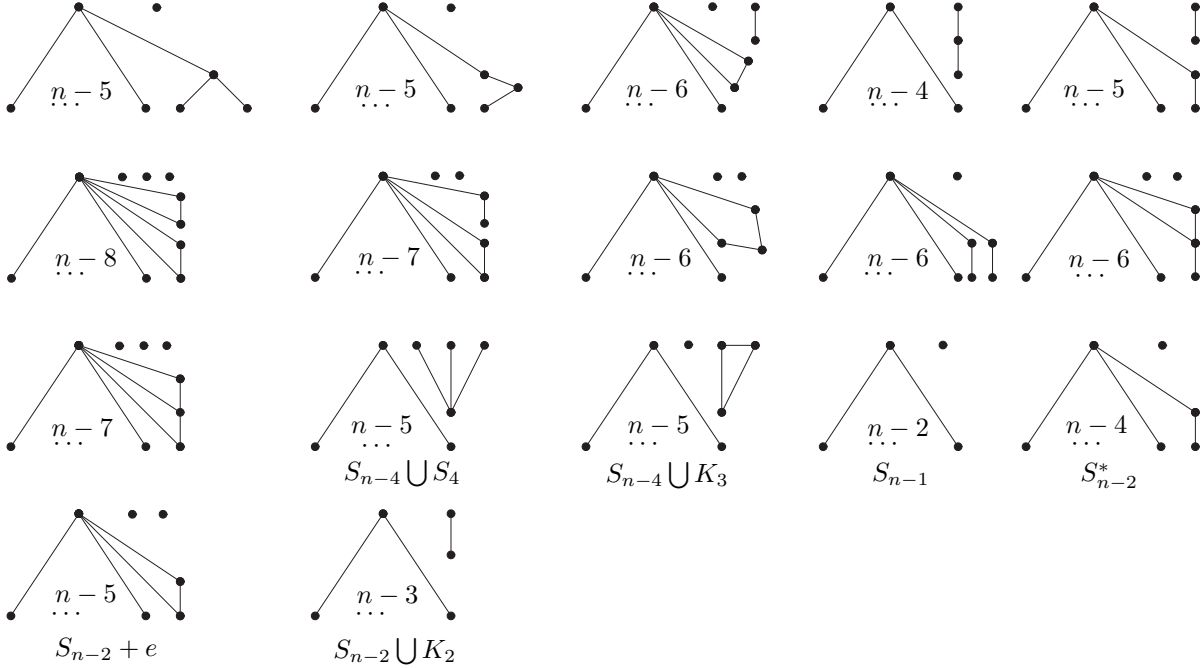} \\
  \caption{The graphs belong to the set $\mathcal{L}_n$.}
\end{center}
\end{figure}
For an integer $k \geqslant 0$, the $k$-closure of the graph $G$ is a graph obtained from $G$ by successively joining pairs of nonadjacent vertices whose degree sum is at least $k$ (in the resulting graph at each stage) until no such pair remains \cite{5, 17}. Write $\mathcal{C}_{k}(G)$ for the $k$-closure of $G$. Note that $d_{\mathcal{C}_{k}(G)}(u)+d_{\mathcal{C}_{k}(G)}(v) \leqslant k-1$ for any pair of nonadjacent vertices $u$ and $v$ in $\mathcal{C}_{k}(G)$.

For short, we omit the isolated vertices of graphs in this paper. For example, in the case $n=6$, $M_{2}$ should be $M_{2} \cup \overline{{K}_{2}}$ with minimum degree one (without considering isolated vertices). Similarly, $S_{n-4} \cup K_{3}$ should be $S_{n-4} \cup K_{3} \cup K_{1}$ with minimum degree one.

A well known theorem of Dirac \cite{dirac} states that if $G$ is a graph on $n$ vertices with $\delta(G)\geq\frac{n}{2}$, then $G$ contains a hamiltonian cycle. In 1963, Pos\'{a} (see also in \cite{erdos} by Erd{\"o}s ) conjectured that if $\delta(G)\geq\frac{2n}{3}$, then $G$ contains a $C_n^2$. Later in 1974, Seymour \cite{sey} generalized Pos\'{a}'s conjecture and conjectured that if $\delta(G)\geq\frac{k}{k+1}n$, then $G$ contains a $C_n^k$. Using the Regularity Lemma and Blow-up Lemma, Koml\'{o}s, S\'{a}rk{\"o}zy and Szemer\'{e}di \cite{k1} proved Seymour's conjecture in asymptotic form, then in \cite{k2,k3} they proved both conjectures for $n\geq n_0$. Later Levitt, S\'{a}rk{\"o}zy and Szemer\'{e}di \cite{L} presented another proof (and a general method) that avoids the use of the Regularity Lemma and thus the resulting $n_0$ is much smaller. Fan and H{\"a}ggkvist \cite{f1} proved that if $\delta(G)\geq\frac{5n}{7}$, then $G$ contains a $C_n^2$. Fan and Kierstead \cite{f2} showed that for every $\epsilon\geq0$, there exists a constant $c$, such that if $\delta(G)\geq(\frac{2}{3}+\epsilon)n+c$, then $G$ contains a $C_n^2$. Next they \cite{f3} showed that if $\delta(G)\geq\frac{2n-1}{3}$, then $G$ contains a $P_n^2$. For more results about the existence of $C_n^k$ in graphs, we refer the reader to \cite{fau,kie}.

A central problem of extremal graph theory is as following: for a given graph $H$, what is the maximum number of edges of an $H$-free graph of order $n$? This question and its extensions are called Tur\'{a}n type problems (for example, \cite{3, 14}). Moreover, much attention has been paid to spectral Tur\'{a}n type problems, i.e., what is the maximum (signless Laplacian, $p$-Laplacian) spectral radius of an $H$-free graph of order $n$ (see for example, \cite{4, 19, 16, 20})? In recent years, some attention has been given to the relations between (signless Laplacian, $p$-Laplacian) spectral radius and cycles of fixed length and particularly Hamilton cycle; see \cite{13, 4, gao, 11, 15, ni, 7}. Yuan \cite{3} determined the maximum number of edges of a graph without containing the $2$-power of a Hamilton cycle which extends a well-known theorem of Ore \cite{14} concerning the maximum number of edges of a graph without containing a Hamilton cycle. Motivated by \cite{3} directly, it is worth to focus on the spectral Tur\'{a}n type problems, i.e., what is the maximum (signless Laplacian, $p$-Laplacian) spectral radius of a $C_n^2$-free graph of order $n$? In this paper, we determine the unique graph with maximum spectral radius
among all graphs of order $n$ without containing the 2-power of a Hamilton cycle. On the other hand, in this paper, we also consider the relationship between the complement of a graph $G$ without containing the $2$-power of a Hamilton cycle and its spectral radius $\mu(\overline{G})$, one may be referred to \cite{4, 2, 20}. Other related results can be found in \cite{12, 9, 10, 8, 19, 1, 6, 18, Libinlong1, Libinlong2, Linhuiqiu1, Linhuiqiu2}.

\subsection{\normalsize Main results}
In this article, we determine the unique graph with maximum spectral radius among all graphs of order $n$ without containing the $2$-power of a Hamilton cycle. This extends a theorem of Long-Tu Yuan \cite{3} in 2021 concerning the maximum number of edges of an $n$-vertex graph without containing the $2$-power of a Hamilton cycle. We will establish the following theorems.

\begin{thm}\label{thm1.1}
Let $G$ be graph on $n \geqslant 18$ vertices. If $e(G)>\frac{n^{2}-3n+3}{2}$, then $G$ contains the $2$-power of a Hamilton cycle unless $G$ is a subgraph of $Y_n$, where $Y_n$ is $K_{n} \backslash E\left(S_{n-3}\right)$, $K_{n} \backslash E\left(S_{n-4} \cup S_4\right)$ or $K_{n} \backslash E\left(S_{n-4} \cup K_3\right)$.
\end{thm}
\begin{thm}\label{thm1.2}
Let $G$ be graph on $n \geqslant 18$ vertices. If $\mu(G) \geqslant \mu\left(K_{n} \backslash E\left(S_{n-3}\right)\right)$, then $G$ contains the $2$-power of a Hamilton cycle unless $G=K_{n} \backslash E\left(S_{n-3}\right)$.
\end{thm}
\begin{thm}\label{thm1.3}
Let $G$ be graph on $n \geqslant 18$ vertices. If $\mu(\overline{G}) \leqslant \sqrt{n-5}$, then $G$ contains the $2$-power of a Hamilton cycle unless $\mathcal{C}_{n}(G)=K_{n}$.
\end{thm}

\subsection{\normalsize Preliminaries}
In this {subsection}, we firs recall some important known results and then present a few technical lemmas which will be used in the proofs of our main results.

\begin{thm}[\cite{3}]\label{thm1.4}
Let $G$ be an $n$-vertex graph with $n \geqslant 6$ and without containing the $2$-power of a Hamilton cycle. Then we have $e(G) \leq \begin{cases}25, ~~\quad\quad\quad n=8; \\ 49, ~~\quad\quad\quad n=11; \\ C_{n-1}^{2}+3, \quad otherwise.\end{cases}$ Moreover, the equality holds if and only if $G=K_{n} \backslash E\left(F\right)$ with $F \in \mathcal{H}_{n}$. where $\mathcal{H}_{n}$ is a family of graphs with $n \geqslant 6$ as follows:
$$\mathcal{H}_{n}=\begin{cases}S_{3}, & n=6; \\ S_{4}, K_{3}, & n=7; \\ K_{3}, & n=8; \\ K_{4}^{-}, S_{6}, & n=9; \\ K_{4}, S_{7}, & n=10; \\ K_{4}, & n=11; \\  S_{9}, & n=12; \\ S_{10}, & n=13; \\ S_{11}, K_{5}, & n=14; \\ S_{n-3}, & n\geqslant15. \end{cases}$$
\end{thm}

\begin{prop}[\cite{3}]\label{prop1.5}
Let $n \geqslant 6$. If $\overline{C_{n-1}^{2}}$ contains a copy of $F$ with $n-1$ vertices, then\\
\rm (i) $\overline{C_{n}^{2}}$ contains each graph in $F^{+}$ as a subgraph.\\
\rm (ii) Let $\left\lceil\frac{n-1}{4}\right\rceil-1=t \geqslant 1$. Then $\overline{C_{n}^{2}}$ contains each graph in $F^{+t}$ as a subgraph.
\end{prop}

\begin{lem}[\cite{3}]\label{lem1.6}
Let $n \geqslant 6$. If $n \neq 8, 11$, then $\overline{C_{n}^{2}}$ contains each copy of $F \in \mathcal{K}(n, n-4) \backslash \mathcal{H}_{n}$. If $n=8, 11$, then $\overline{C_{n}^{2}}$ contains each copy of $F \in \mathcal{K}(n, n-5) \backslash \mathcal{H}_{n}$.
\end{lem}

The following fact is obvious.

\begin{fact}\label{fact1.7}
Let $n \geqslant 6$. If $\frac{n}{3} \in \mathbb{Z}$, then $\overline{C_{n}^{2}}$ contains a copy of $K_{\frac{n}{3}}$ but not $K_{\frac{n}{3}+1}^{-}$. If $\frac{n}{3} \notin \mathbb{Z}$, then $\overline{C_{n}^{2}}$ contains a copy of $K_{\lceil\frac{n}{3}\rceil}^{-}$.
\end{fact}

\begin{lem}[\cite{1}]\label{lem1.8}
Let $G$ be a simple connected graph of order $n$ with $m$ edges. The spectral radius $\mu(G)$ satisfies $\mu(G) \leqslant \sqrt{2m-n+1}$ with equality if and only if $G$ is isomorphic to $S_{n}$ or $K_{n}$.
\end{lem}

\begin{thm}[\cite{4}]\label{thm1.9}
Let $G$ be a graph of order $n$ and spectral radius $\mu(G)$. If $\mu(G)>n-2$, then $G$ contains a Hamilton cycle unless $G=K_{n-1}+e$. Let $\mu(\overline{G})$ be the spectral radius of the complement graph $\overline{G}$. If $\mu(\overline{G}) \leqslant \sqrt{n-2}$, then $G$ contains a Hamilton cycle unless $G=K_{n-1}+e$.
\end{thm}

\begin{lem}[\cite{2}]\label{lem1.10}
Let $G$ be a graph of order $n$. Then the spectral radius $\mu(G)$ of $G$ satisfies $\mu(G) \geqslant \sqrt{\frac{\sum\limits_{u \in V(G)} d_G^{2}(u)}{n}}$ with equality if and only if each component of $G$ is an $r$-regular graph or an $\left(r_{1}, r_{2}\right)$-biregular graph, where $r^{2}=\frac{\sum\limits_{u \in V(G)} d_G^{2}(u)}{n}$ and $r_{1}, r_{2}$ satisfy $r_{1}r_{2}=r^{2}$.
\end{lem}

\begin{lem}\label{lem1.11}
Let $6 \leqslant n \leqslant 15$. Then $\overline{C_{n}^{2}}$ contains each copy of $F \in \mathcal{K}(n, n-3)$ unless $F$ contains one of $\mathcal{F}_{n}$ as a subgraph, where $\mathcal{F}_{n}$ is a family of graphs with $6 \leqslant n \leqslant 15$ as follows:
$$\mathcal{F}_{n}=\left\{\begin{array}{ll}S_{3}, & n=6; \\ S_{4}, C_{4}, K_{3}, & n=7; \\ K_{3}, S_{5}, & n=8; \\ K_{4}^{-}, S_{6}, F_{5}, & n=9; \\ K_{4}, S_{7}, S_{5,2}, & n=10; \\ K_{4}, S_{8}, & n=11; \\ K_{5}^{-}, S_{9}, & n=12; \\ S_{10}, K_{5}, & n=13; \\ S_{11}, K_{5}, & n=14; \\ S_{12}, & n=15. \end{array}\right.$$
\end{lem}
\begin{proof}

Let $n=6$. Then $\overline{C_{6}^{2}}=M_{3}$ and $\mathcal{K}(6,2)=\left\{M_{2}, S_{3}\right\}$. Hence it is easy to see that $\mathcal{K}(6, 3)=\left\{M_{3}, S_{3} \cup K_{2}, P_{4}, S_{4}, K_{3}\right\}$. Then obviously the lemma holds for $n=6$.

Let $n=7$. Then $\overline{C_{7}^{2}}=C_{7}$ and $\mathcal{K}(7, 3)=\left\{M_{3}, S_{3} \cup K_{2}, P_{4}, K_{3}, S_{4}\right\}$. Hence it is easy to see that $\mathcal{K}(7,4)=\left\{M_{2} \cup S_{3}, P_{4} \cup K_{2}, 2S_{3}, P_{5}, K_{3} \cup K_{2}, C_{4}, T_{1, 1, 2}, S_{4} \cup K_{2}, K_{3}^{+}, S_{5}\right\}$. Then obviously the lemma holds for $n=7$.

Let $n=8$. Then clearly $\overline{C_{8}^{2}}$ does not contain a copy of $K_{3}$. It is easy to see that $\mathcal{K}(8,4)=\left\{M_{4}, M_2 \cup S_{3}, P_{4} \cup K_{2}, 2 S_{3}, S_{4} \cup K_{2}, P_{5}, K_{3}^{+}, C_{4}, S_{5}, K_{3} \cup K_{2}, T_{1, 1, 2}\right\}$ and $\mathcal{K}(8,5)=\{M_{2} \cup P_{4}, M_{2} \cup S_{4},$ $S_{3} \cup S_{4}, P_{5} \cup K_{2}, T_{1 ,1 ,2}  \cup K_{2}, C_{4}^{+}, P_{4}  \cup S_{3}, P_{6}, T_{1, 2, 2}, 2 S_{3} \cup K_{2}, T_{1, 1, 3}, D_{6}, C_{4} \cup K_{2}, C_5\} \cup \mathcal{F}_{8}^{\prime}$. Let $\mathcal{F}_{8}^{\prime}=\left\{T_{1, 1, 1, 2}, K_{3}^{+} \cup K_{2}, K_{3} \cup S_{3}, K_{3} \cup M_{2}, S_{5} \cup K_{2}, S_{6}, S_{5}^{*}, K_{4}^{-}, G_{2}, G_{3}, G_{4}\right\}$, where $G_{2}, G_{3}$ and $G_{4}$ are obtained from $K_{3}^{+}$ by adding a new vertex and joining it to a vertex of $K_{3}^{+}$ with degree one, two and three, respectively. It is straightforward to check that $\overline{C_{8}^{2}}$ contains each graph in $\mathcal{K}(8,5) \backslash  \mathcal{F}_{8}^{\prime}$ and the graphs in $\mathcal{F}_{8}^{\prime}$ contains a copy of either $K_{3}$ or $S_{5}$. So the lemma holds for $n=8$.

Let $n=9$. Let $t=\left\lceil\frac{n-1}{4}\right\rceil-1=\left\lceil\frac{9-1}{4}\right\rceil-1=1$. Regardless of isolated vertices of  $F$, we consider the following two cases:\\
(a.1) $\delta(F) \geqslant t+1=2$. Then the number of non-isolated vertices of $F$ is at most $\left\lfloor\frac{2(n-3)}{\left\lceil\frac{n-1}{4}\right\rceil}\right\rfloor=6$.\\
(b.1) $\delta(F)=1$. Then by Proposition \ref{prop1.5}(i), we only need to consider $F \in \mathcal{F}_{8}^{\prime+}$.\\
In case of (a.1) the graphs with $\delta(F) \geqslant 2$ and $e(F)=6$ are $C_{6}, 2 K_{3}, K_{4}, F_{5}, K_{2,3}$ and $C_{5}+e$. A simple observation shows that $\overline{C_{9}^{2}}$ contains $C_{6}, 2 K_{3}, K_{2, 3}$ and $C_{5}+e$ as subgraphs. Now the graphs in (b.1) are the graphs obtained from a graph belonging to $\mathcal{F}_{8}^{\prime}$ by adding an isolated vertex $v$ and an arbitrary edge incident with $v$. Clearly, we can easily check that $\overline{C_{9}^{2}}$ contains a copy of $K_{3}$ but not $K_{4}^{-}$ by Fact \ref{fact1.7}. Thus it is not hard to show that $\overline{C_{9}^{2}}$ contains copies of $\mathcal{F}_{8}^{\prime+}\backslash\left\{S_{6}^{*}, S_{6} \cup K_{2}, S_{7}, S_{6}+e, K_{4}^{-} \cup K_{2}, G_5, G_6\right\}$, where $G_{5}$ and $G_{6}$ are obtained from $K_{4}^{-}$ by adding a new vertex and joining it to a vertex of $K_{4}^{-}$ with degree two and three, respectively. Thus we are done for $n=9$. Moreover, $\overline{C_{9}^{2}}$ contains each graph in $\mathcal{K}(9,5) \backslash\left\{K_{4}^-, S_{6}\right\}$ as a subgraph by Lemma \ref{lem1.6}.

Let $n=10$. Let $t=\left\lceil\frac{n-1}{4}\right\rceil-1=\left\lceil\frac{10-1}{4}\right\rceil-1=2$. Regardless of isolated vertices of $F$, we consider the following three cases:\\
(a.2) $\delta(F) \geqslant t+1=3$. Then the number of non-isolated vertices of $F$ is at most $\left\lfloor\frac{2(n-3)}{\left\lceil\frac{n-1}{4}\right\rceil}\right\rfloor=4$.\\
(b.2) $\delta(F)=2$. Then by Proposition \ref{prop1.5}(ii), we only need to consider $F \in\left\{K_{4}^-, S_{6}\right\}^{+2}$.\\
(c.2) $\delta(F)=1$. Then by Proposition \ref{prop1.5}(i), we only need to consider the graphs obtained from a graph in $\mathcal{K}(9,6)$ which contains one of graph in $\mathcal{F}_9$ as a subgraph by adding an isolated vertex $v$ and an arbitrary edge incident with $v$.\\
Clearly, there is no graph in $(a.2)$. As for the graphs in $(b.2)$, we only need to consider $F \in\left\{K_{4}^{-}\right\}^{+2}$ because of $\delta(F)=2$ without considering isolated vertices. A simple observation shows that $\overline{C_{10}^{2}}$ contains each graph in $\left\{K_{4}^{-}\right\}^{+2} \backslash\left\{S_{5,2}\right\}$ as a subgraph. Clearly, we know that $\overline{C_{10}^{2}}$ contains a copy of $K_{4}^{-}$ but not $K_{4}$ by Fact \ref{fact1.7}. Thus it is not hard to show that $\overline{C_{10}^{2}}$ contains each graph in $(c.2)$ except $K_{4} \cup K_{2}, K_{4}^{+}, S_{7}^{*}, S_{7}\cup K_{2}, S_{8}, S_{7}+e$. Thus we are done for $n=10$. Moreover, $\overline{C_{10}^{2}}$ contains each graph in $\mathcal{K}(10,6) \backslash\left\{K_{4}, S_{7}\right\}$ as a subgraph by Lemma \ref{lem1.6}.

Let $n=11$. Let $t=\left\lceil\frac{n-1}{4}\right\rceil-1=\left\lceil\frac{11-1}{4}\right\rceil-1=2$. Note that $\overline{C_{11}^{2}}$ contains each graph in $\mathcal{K}(11, 7) \backslash\left\{K_4^{+}, K_4 \cup K_2, S_8 \right\}$ by Lemma \ref{lem1.6}. Regardless of isolated vertices of $F$, we consider the following three cases:\\
(a.3) $\delta(F) \geqslant t+1=3$. Then the number of non-isolated vertices of $F$ is at most $\left\lfloor\frac{2(n-3)}{\left\lceil\frac{n-1}{4}\right\rceil}\right\rfloor=5$.\\
(b.3) $\delta(F)=2$. Then by Proposition \ref{prop1.5}(ii), we only need to consider $F \in\left\{K_{4}, S_{7}\right\}^{+2}$. \\
(c.3) $\delta(F)=1$. Then by Proposition \ref{prop1.5}(i), we only need to consider the graphs obtained from a graph in $\mathcal{K}(10,7)$ which contains one of graph in $\mathcal{F}_{10}$ as a subgraph by adding an isolated vertex $v$ and an arbitrary edge incident with $v$.\\
In case of (a.3) the unique graph with $\delta(F) \geqslant 3$ and $e(F)=8$ is $W_{5}$ and $\overline{C_{11}^{2}}$ contains a copy of $W_{5}$. As for the graphs in (b.3), we only need to consider $F \in\left\{K_{4}\right\}^{+2}=K_{4}^{+2}$ because of $\delta(F)=2$.  Clearly, we know that $\overline{C_{11}^{2}}$ contains a copy of $K_{4}^-$ but not $K_{4}$ by Fact \ref{fact1.7} and it is easy to check that $\overline{C_{11}^{2}}$ contains a copy of $S_{5, 2}$. Thus it is not hard to show that $\overline{C_{11}^{2}}$ contains each graph in (c.3) except $S_{8}^{*}, S_{8} \cup K_{2}, S_{9}, S_{8}+e,  K_{4} \cup M_{2}, K_{4} \cup S_{3}, K_{4}^{+} \cup K_{2}, G_{7}, G_{8}$ and $G_{9}$, where $G_{7}, G_{8}$ and $G_{9}$ are obtained from $K_{4}^{+}$ by adding a new vertex and joining it to a vertex of $K_{4}^{+}$ with degree one, three and four respectively. Thus we are done for $n=11$.

Let $n=12$. Let $t=\left\lceil\frac{n-1}{4}\right\rceil-1=\left\lceil\frac{12-1}{4}\right\rceil-1=2$. Regardless of isolated vertices of $F$, we consider the following three cases:\\
(a.4) $\delta(F) \geqslant t+1=3$. Then the number of non-isolated vertices of $F$ is at most $\left\lfloor\frac{2(n-3)}{\left\lceil\frac{n-1}{4}\right\rceil}\right\rfloor=6$.\\
(b.4) $\delta(F)=2$. Then by Proposition \ref{prop1.5}(ii), we only need to consider $F \in\left\{K_{4}^{+}, K_{4} \cup K_{2}, S_{8}\right\}^{+2}$.\\
(c.4) $\delta(F)=1$. Then by Proposition \ref{prop1.5}(i), we only need to consider the graphs obtained from a graph in $\mathcal{K}(11,8)$ which contains one of graph in $\mathcal{F}_{11}$ as a subgraph by adding an isolated vertex $v$ and an arbitrary edge incident with $v$.\\
In case of (a.4), the graphs with $\delta(F) \geqslant 3$ and $e(F)=9$ are $K_{5}^{-},K_{3, 3}$ and $G_{10}$, where $G_{10}$ is obtained from two vertex disjoint copies of $K_{3}$ and joining three independent edges between them. Clearly, we know that $\overline{C_{12}^{2}}$ contains a copy of $K_{4}$ but not $K_{5}^{-}$ by Fact \ref{fact1.7} and it is easy to check that $\overline{C_{12}^{2}}$ contains copies of $K_{3, 3}$ and $G_{10}$. As for the graphs in (b.4), we only need to consider $F \in\left\{K_{4}^{+}, K_{4} \cup K_{2}\right\}^{+2}$ because of $\delta(F)=2$ and also it is not hard to see that $\overline{C_{12}^{2}}$ contains each graph in $\left\{K_{4}^{+}, K_{2} \cup K_{4}\right\}^{+2}$. Thus it is not hard to show that $\overline{C_{12}^{2}}$ contains each graph in (c.4) except $S_{9}^{*}, S_{9} \cup K_{2}, S_{10}$ and $S_{9}+e$. Thus we are done for $n=12$. Moreover, $\overline{C_{12}^{2}}$ contains each graph in $\mathcal{K}(12, 8) \backslash\left\{S_{9}\right\}$ as a subgraph by Lemma \ref{lem1.6}.

Let $n=13$. Let $t=\left\lceil\frac{n-1}{4}\right\rceil-1=\left\lceil\frac{13-1}{4}\right\rceil-1=2$. Regardless of isolated vertices of $F$, we consider the following three cases:\\
(a.5) $\delta(F) \geqslant t+1=3$. Then the number of non-isolated vertices of $F$ is at most $\left\lfloor\frac{2(n-3)}{\left\lceil\frac{n-1}{4}\right\rceil}\right\rfloor=6$.\\
(b.5) $\delta(F)=2$. Then by Proposition \ref{prop1.5}(ii), we only need to consider $F \in\left\{S_{9}\right\}^{+2}$.\\
(c.5) $\delta(F)=1$. Then by Proposition \ref{prop1.5}(i), we only need to consider the graphs obtained from a graph in $\mathcal{K}(12,9)$ which contains one of graph in $\mathcal{F}_{12}$ as a subgraph by adding an isolated vertex $v$ and an arbitrary edge incident with $v$.\\
In case of (a.5), the graphs with $\delta(F) \geqslant 3$ and $e(F)=10$ are $K_{5}, W_{6}, G_{11}, G_{12}$ and $G_{13}$, where $G_{11}, G_{12}$ and $G_{13}$ are depicted in Figure 3.
Clearly, we know that $\overline{C_{13}^{2}}$ contains a copy of ${K_{5}^-}$ but not $K_{5}$ by Fact \ref{fact1.7} and it is not hard to check that $\overline{C_{13}^{2}}$ contains copies of $W_{6}, G_{11}, G_{12}$ and $G_{13}$. As for the graphs in (b.5), there is no graph in (b.5) because of $\delta(F)=2$. Thus it is not hard to show that $\overline{C_{13}^{2}}$ contains each graph in (c.5) except $S_{10}^{*}, S_{10} \cup K_{2}, S_{11}$ and $S_{10}+e$. Thus we are done for $n=13$. Moreover, $\overline{C_{13}^{2}}$ contains each graph in $\mathcal{K}(13,9) \backslash\left\{S_{10}\right\}$ as a subgraph by Lemma \ref{lem1.6} and obviously $\overline{C_{13}^{2}}$ contains each graph in $\mathcal{K}(13, 8)$ as a subgraph.
\begin{figure}[h!]\label{fig3}
\begin{center}
\psfrag{1}{$G_{11}$}\psfrag{2}{$G_{12}$}\psfrag{3}{$G_{13}$}
\includegraphics[width=120mm]{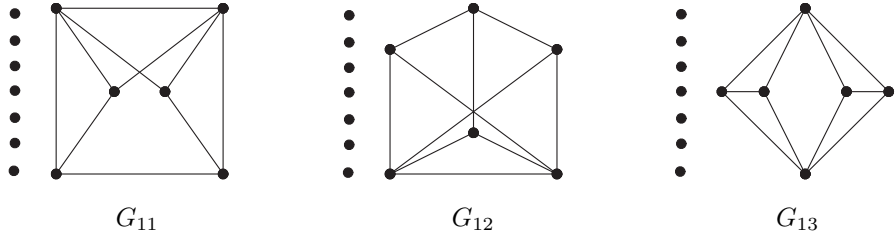} \\
  \caption{$G_{11}, G_{12}$ and $G_{13}$.}
\end{center}
\end{figure}

Let $n=14$. Let $t=\left\lceil\frac{n-1}{4}\right\rceil-1=\left\lceil\frac{14-1}{4}\right\rceil-1=3$. Regardless of isolated vertices of $F$, we consider the following four cases:\\
(a.6) $\delta(F) \geqslant t+1=4$. Then the number of non-isolated vertices of $F$ is at most $\left\lfloor\frac{2(n-3)}{\left\lceil\frac{n-1}{4}\right\rceil}\right\rfloor=5$.\\
(b.6) $\delta(F)=3$. Then by Proposition \ref{prop1.5}(ii), $\overline{C_{14}^{2}}$ contains each copy of $F \in \mathcal{K}(13, 8)^{+3}$.\\
(c.6) $\delta(F)=2$. Then by Proposition \ref{prop1.5}(ii), we only need to consider $F \in\left\{S_{10}\right\}^{+2}$.\\
(d.6) $\delta(F)=1$. Then by Proposition \ref{prop1.5}(i), we only need to consider the graphs obtained from a graph in $\mathcal{K}(13,10)$ which contains one of graph in $\mathcal{F}_{13}$ as a subgraph by adding an isolated vertex $v$ and an arbitrary edge incident with $v$.\\
In case of (a.6) there does not exist a graph with $\delta(F) \geqslant 4$ and $e(F)=11$. As for the graphs in (c.6), there is no graph in (c.6) because of $\delta(F)=2$. Clearly, we know that $\overline{C_{14}^{2}}$ contains a copy of $K_{5}^-$ but not $K_{5}$ by Fact \ref{fact1.7}. Thus it is not hard to show that $\overline{C_{14}^{2}}$ contains each graph in (d.6) except $K_{5}^{+}, K_{5} \cup K_{2}, S_{11}^{*}, S_{11} \cup K_{2}, S_{12}$ and $S_{11}+e$. Thus we are done for $n=14$. Moreover, $\overline{C_{14}^{2}}$ contains each graph in $\mathcal{K}(14, 10) \backslash\left\{S_{11}, K_5\right\}$ as a subgraph by Lemma \ref{lem1.6} and obviously $\overline{C_{14}^{2}}$ contains each graph in $\mathcal{K}(14, 9)$ as a subgraph.

Let $n=15$. Let $t=\left\lceil\frac{n-1}{4}\right\rceil-1=\left\lceil\frac{15-1}{4}\right\rceil-1=3$. Regardless of isolated vertices of $F$, we consider the following four cases:\\
(a.7) $\delta(F) \geqslant t+1=4$. Then the number of non-isolated vertices of $F$ is at most $\left\lfloor\frac{2(n-3)}{\left\lceil\frac{n-1}{4}\right\rceil}\right\rfloor=6$.\\
(b.7) $\delta(F)=3$. Then by Proposition \ref{prop1.5}(ii), $\overline{C_{15}^{2}}$ contains each copy of $F \in \mathcal{K}(14, 9)^{+3}$.\\
(c.7) $\delta(F)=2$. Then by Proposition \ref{prop1.5}(ii), we only need to consider $F \in\left\{S_{11},K_5\right\}^{+2}$.\\
(d.7) $\delta(F)=1$. Then by Proposition \ref{prop1.5}(i), we only need to consider the graphs obtained from a graph in $\mathcal{K}(14,11)$ which contains one of graph in $\mathcal{F}_{14}$ as a subgraph by adding an isolated vertex $v$ and an arbitrary edge incident with $v$.\\
In case of (a.7) the unique graph with $\delta(F) \geqslant 4$ and $e(F)=12$ is $\overline{K_{2}} \vee C_{4}$ and it is easy to check that $\overline{C_{15}^{2}}$ contains a copy of $\overline{K_{2}} \vee C_{4}$. As for the graphs in (c.7), we only need to consider $F \in\left\{K_{5}\right\}^{+2}=K_{5}^{+2}$ because of $\delta(F)=2$. Clearly, we know that $\overline{C_{15}^{2}}$ contains a copy of $K_{5}$ by Fact \ref{fact1.7} and obviously contains a copy of $K_{5}^{+2}$. Thus it is not hard to show that $\overline{C_{15}^{2}}$ contains each graph in (d.7) except for $S_{12}^{*}, S_{12} \cup K_{2}, S_{13}$ and $S_{12}+e$. Thus we are done for $n=15$.
\end{proof}

\begin{lem}\label{lem1.12}
Let $n \geqslant 15$. Then $\overline{C_{n}^{2}}$ contains each copy of $F \in \mathcal{K}(n, n-3)$ unless $F$ contains $S_{n-3}$ as a subgraph.
\end{lem}

\begin{proof}
We prove the lemma by induction on $n$. For $n=15$, the lemma follows from Lemma \ref{lem1.11}. Suppose it is true for $n-1$. Then by induction hypothesis, $\overline{C_{n-1}^{2}}$ contains each copy of $F \in \mathcal{K}(n-1, n-4)$ unless $F$ contains $S_{n-4}$ as a subgraph. Moreover, $\overline{C_{n-1}^{2}}$ contains each graph in $\mathcal{K}(n-1, n-5) \backslash\left\{S_{n-4}\right\}$ and contains each graph in $\mathcal{K}(n-1, n-6)$ by Lemma \ref{lem1.6}.

It is now sufficient to show that $\overline{C_{n}^{2}}$ contains each copy of $F \in \mathcal{K}(n, n-3)$ unless $F$ contains $S_{n-3}$ as a subgraph. Let $t=\left\lceil\frac{n-1}{4}\right\rceil-1$. Regardless of isolated vertices of $F$, we consider the following two cases:\\
(a) $1 \leq \delta(F) \leq t$. Applying Proposition \ref{prop1.5}, it is easy to see that $\overline{C_{n}^{2}}$ contains each copy of $F$ unless $F \in\left\{S_{n-3} \cup K_{2}, S_{n-3}^{*}, S_{n-3}+e, S_{n-2}\right\}$, in which each graph contains $S_{n-3}$ as a subgraph.\\
(b) $\delta(F) \geqslant t+1$. Then the number of non-isolated vertices of $F$ is at most $\left\lfloor\frac{2(n-3)}{\left\lceil\frac{n-1}{4}\right\rceil}\right\rfloor$.\\
We know that $\overline{C_{n}^{2}}$ contains a copy of $K_{\lfloor\frac{n}{3}\rfloor}$ by Fact \ref{fact1.7}. Now, if $n \geqslant 22$, then we have $\left\lfloor\frac{2(n-3)}{\left\lceil\frac{n-1}{4}\right\rceil}\right\rfloor \leqslant\left\lfloor\frac{n}{3}\right\rfloor$. Thus $\overline{C_{n}^{2}}$ contains a copy of $F$ for $n \geqslant 22$. Let $n=16$. Then the number of non-isolated vertices of $F$ is at most $\left\lfloor\frac{2(n-3)}{\left\lceil\frac{n-1}{4}\right\rceil}\right\rfloor=6$ and $e(F)=13$. By Fact \ref{fact1.7}, $\overline{C_{16}^{2}}$ contains a copy of $K_{6}^-$ which has 14 edges and one can easily check that $\overline{C_{16}^{2}}$ contains a copy of $F$. Let $n=17$. Then the number of non-isolated vertices of $F$ is at most $\left\lfloor\frac{2(n-3)}{\left\lceil\frac{n-1}{4}\right\rceil}\right\rfloor=7$ and $e(F)=14$. Actually $F$ must be one of $\left\{K_{6}^{-}, G_{14}, G_{15}\right\}$ because $\delta(F) \geqslant\left\lceil\frac{n-1}{4}\right\rceil=4$, where $G_{14}$ and $G_{15}$ are depicted in Figure 4. Clearly, we know that $\overline{C_{17}^{2}}$ contains a copy of $K_{6}^{-}$ by Fact \ref{fact1.7} and it is not hard to check that $\overline{C_{17}^{2}}$ contains copies of $G_{14}$ and $G_{15}$. So $\overline{C_{17}^{2}}$ contains a copy of $F$. Let $n=18$. Then the number of non-isolated vertices of $F$ is at most $\left\lfloor\frac{2(n-3)}{\left\lceil\frac{n-1}{4}\right\rceil}\right\rfloor=6$ and $e(F)=15$. Thus $F=K_{6}$ and $\overline{C_{18}^{2}}$ contains a copy of $K_{6}$ by Fact \ref{fact1.7}. Let $n=19$ or $n=20$. Then the number of non-isolated vertices of $F$ is at most $\left\lfloor\frac{2(19-3)}{\left\lceil\frac{19-1}{4}\right\rceil}\right\rfloor=\left\lfloor\frac{2(20-3)}{\left\lceil\frac{20-1}{4}\right\rceil}\right\rfloor=6$ and $e(F)=16$ or $17$. Clearly, such $F$ does not exist. Let $n=21$. Then the number of non-isolated vertices of $F$ is at most $\left\lfloor\frac{2(n-3)}{\left\lceil\frac{n-1}{4}\right\rceil}\right\rfloor=7$ and $e(F)=18$. By Fact \ref{fact1.7}, $\overline{C_{21}^{2}}$ contains a copy of $K_{7}$ which has 21 edges and one can easily check that $\overline{C_{21}^{2}}$ contains a copy of $F$.

This completes the proof.
\end{proof}
\begin{figure}[h!]\label{fig4}
\begin{center}
\psfrag{1}{$G_{14}$}\psfrag{2}{$G_{15}$}
\includegraphics[width=90mm]{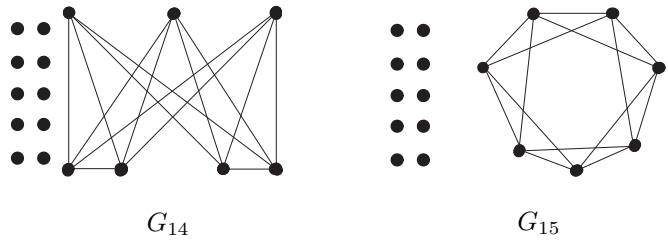} \\
  \caption{$G_{14}$ and $G_{15}$.}
\end{center}
\end{figure}

\begin{lem}\label{lem1.13}
Let $6 \leqslant n \leqslant 18$. Then $\overline{C_{n}^{2}}$ contains each copy of $F \in \mathcal{K}(n, n-2)\backslash \left\{S_{n-4} \cup S_4, S_{n-4} \cup K_3\right\}$ unless $F$ contains one of $\mathcal{E}_{n}$ as a subgraph, where $\mathcal{E}_{n}$ is a family of graphs with $6 \leqslant n \leqslant 18$ as follows:
$$\mathcal{E}_{n}=\begin{cases}S_{3}, & n=6; \\ S_{4}, C_{4}, K_{3}, C_{5}, & n=7; \\ K_{3}, S_{5}, K_{2,3}, & n=8; \\ K_{4}^{-}, S_{6}, F_{5}, & n=9; \\ K_{4}, S_{7}, S_{5,2}, W_5, G_1, & n=10; \\ K_{4}, S_{8}, S_{6,2}, & n=11; \\ K_{5}^{-}, S_{9}, & n=12; \\ S_{10}, K_{5}, & n=13; \\ S_{11}, K_{5}, & n=14; \\ S_{12}, & n=15; \\ S_{13}, & n=16; \\ S_{14}, K_6, & n=17; \\ S_{15}, & n=18, \end{cases}$$ where $G_1$  is depicted in Figure 5.
\end{lem}
\begin{figure}[h!]\label{fig5}
\begin{center}
\psfrag{1}{$G_{1}$}
\includegraphics[width=30mm]{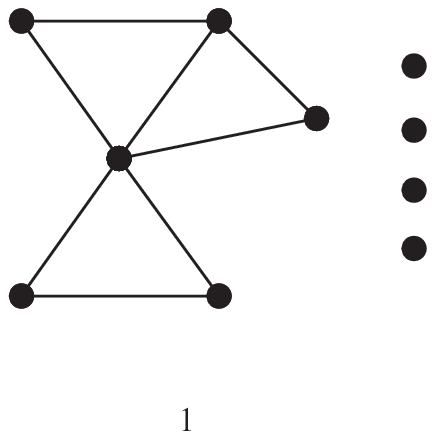} \\
  \caption{$G_{1}$.}
\end{center}
\end{figure}
\begin{proof}
Let $n=6$. Then $\overline{C_{6}}=M_{3}$ and $\mathcal{K}(6,3)=\left\{M_{3}, S_{3} \cup K_{2}, P_{4}, S_{4}, K_{3}\right\}$. Hence it is easy to see that each graph in $\mathcal{K}(6,4)$ contains a copy of $S_{3}$. Then the lemma holds for $n=6$.

Let $n=7$. Then $\overline{C_{7}^{2}}=C_{7}$ and $\mathcal{K}(7,4)=\{M_{2} \cup S_{3}, P_{4} \cup K_{2}, 2 S_{3}, P_{5}, K_{3} \cup K_{2}, C_{4}, T_{1, 1, 2}, S_{4} \cup K_{2}, K_{3}^{+}$, $S_{5}\}$. Clearly, now we only need to consider $\left\{M_{2} \cup  S_{3}, P_{4} \cup K_{2} \cup K_{1}, 2S_{3} \cup K_{1}, P_{5} \cup \overline{K_{2}}\right\}+e$. It is not hard to show that $\overline{C_{7}^{2}}$ contains copies of the above-mentioned graphs except $T_{1, 1, 3}, T_{1,2,2}, G_{2}, G_{3}, C_{4}^{+}, C_{5}, S_{4} \cup S_{3}, D_{6}, K_{3} \cup S_{3}, C_{4} \cup  K_{2}, K_{3}^{+} \cup  K_{2}, K_{3} \cup M_{2}$ and $T_{1,2,2} \cup  K_{2}$. Then obviously the lemma holds for $n=7$.

Let $n=8$. Then clearly $\overline{C_{8}^{2}}$ does not contain copies of $K_{3}$ and $K_{2,3}$ but contains copies of $D_{6}$ and $D_{7}$. By Lemma \ref{lem1.11}, we know that $\mathcal{K}(8.5)=\{M_{2} \cup P_{4}, M_{2} \cup S_{4}, S_{3} \cup S_{4} \cup K_{1}, P_5 \cup K_{2} \cup K_{1}, T_{1, 1, 2}$ $\cup K_{2} \cup K_{1}, C_{4}^{+} \cup \overline{K_{3}}, P_4 \cup S_{3} \cup K_{1}, P_{6} \cup \overline{K_{2}}, T_{1,2,2} \cup\overline{K_{2}}, 2 S_{3} \cup K_{2}, T_{1, 1, 3} \cup \overline{K_{2}}, D_{6} \cup \overline{K_{2}}, C_{4} \cup K_{2} \cup \overline{K_{2}}$, $C_5 \cup \overline{K_{3}}\} \cup \mathcal{F}_{8}^{\prime}$ and obviously each graph in $\mathcal{F}_{8}^{\prime}$ contains one of $\mathcal{E}_{8}$ as a subgraph. Clearly, now we only need to consider $\mathcal{K}(8,5) \backslash \mathcal{F}_{8}^{\prime}+e$. It is straightforward to check that $\overline{C_{8}^{2}}$ contains copies of the above-mentioned graphs except $T_{1, 1, 1, 2} \cup K_{2}, T_{1, 1, 2, 2}, T_{1, 1, 1, 3}, G_{16}, G_{17}, S_{5}^{*} \cup K_{2}, S_{5} \cup S_{3}, 2S_{4}, K_{2,3}, K_{3} \cup S_{4}, K_{3} \cup P_{4}, K_{3} \cup S_{3} \cup K_{2}, K_{3}^{+} \cup M_{2}, K_{3}^{+} \cup S_{3}, G_{5}, G_{6}, G_{2} \cup K_{2}, G_{3} \cup K_{2}, G_{4} \cup K_{2}, G_{18}, G_{19}, G_{20}, G_{21}, G_{22}, G_{23}, C_{5}+e$ and $K_{4}^{-} \cup K_{2}$, where $G_{18}-G_{23}$ are depicted in Figure 6, $G_{16}$ is obtained from $D_{6}$ by adding a new vertex and joining it to a vertex of $D_{6}$ with degree three and $G_{17}$ is obtained from $C_{4}^{+}$ by adding a new vertex and joining it to a vertex of $C_{4}^{+}$ with degree three. So the lemma holds for $n=8$.
\begin{figure}[h!]\label{fig6}
\begin{center}
\psfrag{8}{$G_{18}$}\psfrag{9}{$G_{19}$}\psfrag{0}{$G_{20}$}\psfrag{1}{$G_{21}$}\psfrag{2}{$G_{22}$}\psfrag{3}{$G_{23}$}
\includegraphics[width=150mm]{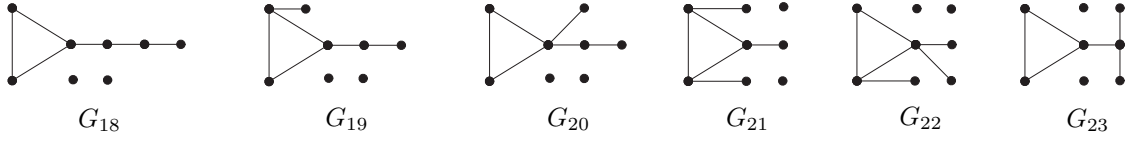} \\
  \caption{$G_{18}-G_{23}$.}
\end{center}
\end{figure}

For $n \geqslant 9$, it is clear that the graphs in $\mathcal{K}(n, n-2)$ containing one of $S_{n-3}$, $S_{n-4} \cup S_{4}$ and $S_{n-4} \cup K_{3}$ as a subgraph form the set $\mathcal{L}_{n}$.

Let $n=9$. Let $t=\left\lceil\frac{n-1}{4}\right\rceil-1=\left\lceil\frac{9-1}{4}\right\rceil-1=1$. Regardless of isolated vertices of $F$, we consider the following two cases:\\
(a.1) $\delta(F) \geqslant t+1=2$. Then the number of non-isolated vertices of $F$ is at most $\left\lfloor\frac{2(n-2)}{\left\lceil\frac{n-1}{4}\right\rceil}\right\rfloor=7$.\\
(b.1) $\delta(F)=1$. Then by Proposition \ref{prop1.5}(i), we only need to consider the graphs obtained from a graph $E$ in $\mathcal{K}(8, 6)$ which contains one of graphs in $\mathcal{E}_{8} \cup\left\{2S_{4}\right\}$ as a subgraph by adding an isolated vertex $v$ and an arbitrary edge incident with $v$.\\
In case of (a.1) the graphs with $\delta(F) \geqslant 2$ and $e(F)=7$ are $S_{5,2}, K_{1} \vee P_{4}, G_{24}, G_{25}, G_{26}, G_{27}, G_{28}, G_{29}, C_{7}$ and $K_{3} \cup C_{4}$, where $G_{24}-G_{29}$ are depicted in Figure 7. A simple observation shows that $\overline{C_{9}^{2}}$ contains each graph of $K_{2,3}, C_{5}, C_{6}, C_{7}$ and $K_{3}$ as a subgraph but does not contain copies of $K_{4}^-$ and $F_{5}$. Thus it is easy to check that $\overline{C_{9}^{2}}$ contains each graph in (a.1) except $S_{5,2}, K_{1} \vee P_{4}$ and $G_{24}$ which obviously contain a copy of $K_{4}^{-}$. As for the graphs in (b.1), note that if the graph $E$ contains one of graphs in $\mathcal{E}_{9}$ as a subgraph, then we do not need to consider $E^{+}$. Thus, without considering the above situations, it is not hard to show that $\overline{C_9^{2}}$ contains each graph left in (b.1) except those graphs in $\mathcal{L}_{9} \backslash\left\{G_{30}, G_{31}, S_{7}^{*}, S_{7} \cup K_{2}, S_{8}, S_{7}+e\right\}$, where $G_{30}$ is obtained from $F_{5}$ by adding a new vertex and joining it to a vertex of $F_{5}$ with degree four and $G_{31}$ is obtained from $G_{6}$ by adding a new vertex and joining it to a vertex of $G_{6}$ with degree four. So we are done for $n=9$. Moreover, $\overline{C_{9}^{2}}$ contains each copy of $F \in \mathcal{K}(9,6)$ unless $F$ contains one of graphs in $\mathcal{F}_{9}$ as a subgraph by Lemma \ref{lem1.11}.
\begin{figure}[h!]\label{fig7}
\begin{center}
\psfrag{4}{$G_{24}$}\psfrag{5}{$G_{25}$}\psfrag{6}{$G_{26}$}\psfrag{7}{$G_{27}$}\psfrag{8}{$G_{28}$}\psfrag{9}{$G_{29}$}
\includegraphics[width=150mm]{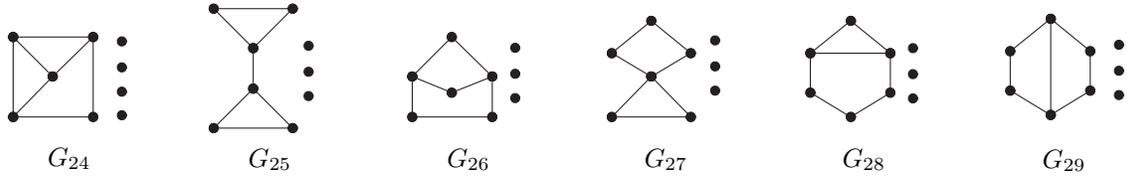} \\
  \caption{$G_{24}-G_{29}$.}
\end{center}
\end{figure}

Let $n=10$. Let $t=\left\lceil\frac{n-1}{4}\right\rceil-1=\left\lceil\frac{10-1}{4}\right\rceil-1=2$. Regardless of isolated vertices of $F$, we consider the following three cases:\\
(a.2) $\delta(F) \geqslant t+1=3$. Then the number of non-isolated vertices of $F$ is at most $\left\lfloor\frac{2(n-2)}{\left\lceil\frac{n-1}{4}\right\rceil}\right\rfloor=5$.\\
(b.2) $\delta(F)=2$. Then by Proposition \ref{prop1.5}(ii), we only need to consider the graphs obtained from a graph $D$ in $\mathcal{K}(9, 6)$ which contains one of graphs in $\mathcal{F}_{9}$ as a subgraph by adding an isolated vertex $v$ and two arbitrary edges incident with $v$.\\
(c.2) $\delta(F)=1$. Then by Proposition \ref{prop1.5}(i), we only need to consider the graphs obtained from a graph $E$ in $\mathcal{K}(9, 7)$ which contains one of graphs in $\mathcal{E}_{9} \cup \left\{ S_{5} \cup S_{4}, S_{5} \cup K_{3}\right\}$ as a subgraph by adding an isolated vertex $v$ and an arbitrary edge incident with $v$.\\
In case of (a.2) the unique graph with $\delta(F) \geqslant 3$ and $e(F)=8$ is $W_{5}$ which belongs to $\mathcal{E}_{10}$. Thus there is no graph in (a.2). A simple observation shows that $\overline{C_{10}^{2}}$ contains each graph of $K_{4}^{-}$and $F_{5}$ as a subgraph but does not contain copies of $K_{4}, S_{5, 2}, G_{1}$ and $W_{5}$. As for the graphs in (b.2), note that if the graph $D$ contains a copy of $S_{6}$, then we do not need to consider $D^{+2}$ because of $\delta(F)=2$. Moreover, we do not need to consider the graph $D=K_{4}$ belonging to $\mathcal{E}_{10}$. Thus it is easy to check that $\overline{C_{10}^{2}}$ contains each graph left in (b.2) except $G_{1}$ which belongs to $\mathcal{E}_{10}$. As for the graphs in (c.2), note that if the graph $E$ contains one of graphs in $\mathcal{E}_{10}$ as a subgraph, then we do not need to consider $E^{+}$. Thus, without considering the above situations, it is not hard to show that $\overline{C_{10}^{2}}$ contains each graph left in (c.2) except those graphs in $\mathcal{L}_{10} \backslash \left\{S_{8}^{*}, S_{8} \cup K_{2}, S_{9}, S_{8}+e\right\}$. So we are done for $n=10$. Moreover, $\overline{C_{10}^{2}}$ contains each copy of $F \in \mathcal{K}(10,7)$ unless $F$ contains one of graphs in $\mathcal{F}_{10}$ as a subgraph by Lemma \ref{lem1.11}.

Let $n=11$. Let $t=\left\lceil\frac{n-1}{4}\right\rceil-1=\left\lceil\frac{11-1}{4}\right\rceil-1=2$. Regardless of isolated vertices of $F$, we consider the following three cases:\\
(a.3) $\delta(F) \geqslant t+1=3$. Then the number of non-isolated vertices of $F$ is at most $\left\lfloor\frac{2(n-2)}{\left\lceil\frac{n-1}{4}\right\rceil}\right\rfloor=6$.\\
(b.3) $\delta(F)=2$. Then by Proposition \ref{prop1.5}(ii), we only need to consider the graphs obtained from a graph $D$ in $\mathcal{K}(10, 7)$ which contains one of graphs in $\mathcal{F}_{10}$ as a subgraph by adding an isolated vertex $v$ and two arbitrary edges incident with $v$.\\
(c.3) $\delta(F)=1$. Then by Proposition \ref{prop1.5}(i), we only need to consider the graphs obtained from a graph $E$ in $\mathcal{K}(10, 8)$ which contains one of graphs in $\mathcal{E}_{10} \cup \left\{S_{6} \cup S_{4}, S_{6} \cup K_{3}\right\}$ as a subgraph by adding an isolated vertex $v$ and an arbitrary edge incident with $v$.\\
In case of (a.3) the graphs with $\delta(F) \geqslant 3$ and $e(F)=9$ are $K_{5}^{-}, K_{3, 3}$ and $G_{10}$. A simple observation shows that $\overline{C_{11}^{2}}$ contains each graph of $K_{4}^{-}, K_{3, 3}, G_{10}, S_{5, 2}, W_{5}$ and $G_{1}$ as a subgraph but does not contain copies of $K_{4}$ and $S_{6,2}$. As for the graphs in (b.3), note that if the graph $D$ contains a copy of $S_{7}$, then we do not need to consider $D^{+2}$ because of $\delta(F)=2$. Moreover, we do not need to consider the graph $D$ which contains a copy of $K_{4}$ belonging to $\mathcal{E}_{11}$, i.e. $D$ is either $K_{4}^{+}$ or $K_{4} \cup K_{2}$ under the circumstance. Thus it is easy to check that $\overline{C_{11}^{2}}$ contains each graph left in (b.3) except $S_{6, 2}$ which belongs to $\mathcal{E}_{11}$. As for the graphs in (c.3), note that if the graph $E$ contains one of graphs in $\mathcal{E}_{11}$ as a subgraph, then we do not need to consider $E^{+}$. Thus, without considering the above situations, it is not hard to show that $\overline{C_{11}^{2}}$ contains each graph left in (c.3) except those graphs in $\mathcal{L}_{11}\backslash \left\{S_{9}^{*}, S_{9} \cup K_{2}, S_{10}, S_{9}+e\right\}$. So we are done for $n=11$. Moreover, $\overline{C_{11}^{2}}$ contains each copy of $F \in \mathcal{K}(11, 8)$ unless $F$ contains one of graphs in $\mathcal{F}_{11}$ as a subgraph by Lemma \ref{lem1.11}.

Let $n=12$. Let $t=\left\lceil\frac{n-1}{4}\right\rceil-1=\left\lceil\frac{12-1}{4}\right\rceil-1=2$. Regardless of isolated vertices of $F$, we consider the following three cases:\\
(a.4) $\delta(F) \geqslant t+1=3$. Then the number of non-isolated vertices of $F$ is at most $\left\lfloor\frac{2(n-2)}{\left\lceil\frac{n-1}{4}\right\rceil}\right\rfloor=6$.\\
(b.4) $\delta(F)=2$. Then by Proposition \ref{prop1.5}(ii), we only need to consider the graphs obtained from a graph $D$ in $\mathcal{K}(11, 8)$ which contains one of graphs in $\mathcal{F}_{11}$ as a subgraph by adding an isolated vertex $v$ and two arbitrary edges incident with $v$.\\
(c.4) $\delta(F)=1$. Then by Proposition \ref{prop1.5}(i), we only need to consider the graphs obtained from a graph $E$ in $\mathcal{K}(11, 9)$ which contains one of graphs in $\mathcal{E}_{11} \cup \left\{S_{7} \cup S_{4}, S_{7} \cup K_{3}\right\}$ as a subgraph by adding an isolated vertex $v$ and an arbitrary edge incident with $v$.\\
In case of (a.4) the graphs with $\delta(F) \geqslant 3$ and $e(F)=10$ are $K_{5}, W_{6}, G_{11}, G_{12}$ and $G_{13}$. A simple observation shows that $\overline{C_{12}^{2}}$ contains each graph of $K_{4}, W_{6}, S_{6,2}, G_{11}, G_{12}$ and $G_{13}$ as a subgraph but does not contain a copy of $K_{5}^{-}$. As for the graphs in (b.4), note that if the graph $D$ contains a copy of $S_{8}$, then we do not need to consider $D^{+2}$ because of $\delta(F)=2$. Thus it is easy to check that $\overline{C_{12}^{2}}$ contains each graph left in (b.4). As for the graphs in (c.4), note that if the graph $E$ contains one of graphs in $\mathcal{E}_{12}$ as a subgraph, then we do not need to consider $E^{+}$. Thus, without considering the above situations, it is not hard to show that $\overline{C_{12}^{2}}$ contains each graph left in (c.4) except those graphs in $\mathcal{L}_{12}\backslash \left\{S_{10}^{*}, S_{10} \cup K_{2}, S_{11}, S_{10}+e\right\}$. So we are done for $n=12$. Moreover, $\overline{C_{12}^{2}}$ contains each copy of $F \in \mathcal{K}(12, 9)$ unless $F$ contains one of graphs in $\mathcal{F}_{12}$ as a subgraph by Lemma \ref{lem1.11}.

Let $n=13$. Let $t=\left\lceil\frac{n-1}{4}\right\rceil-1=\left\lceil\frac{13-1}{4}\right\rceil-1=2$. Regardless of isolated vertices of $F$, we consider the following three cases:\\
(a.5) $\delta(F) \geqslant t+1=3$. Then the number of non-isolated vertices of $F$ is at most $\left\lfloor\frac{2(n-2)}{\left\lceil\frac{n-1}{4}\right\rceil}\right\rfloor=7$.\\
(b.5) $\delta(F)=2$. Then by Proposition \ref{prop1.5}(ii), we only need to consider the graphs obtained from a graph $D$ in $\mathcal{K}(12, 9)$ which contains one of graphs in $\mathcal{F}_{12}$ as a subgraph by adding an isolated vertex $v$ and two arbitrary edges incident with $v$.\\
(c.5) $\delta(F)=1$. Then by Proposition \ref{prop1.5}(i), we only need to consider the graphs obtained from a graph $E$ in $\mathcal{K}(12, 10)$ which contains one of graphs in $\mathcal{E}_{12} \cup \left\{S_{8} \cup S_{4}, S_{8} \cup K_{3}\right\}$ as a subgraph by adding an isolated vertex $v$ and an arbitrary edge incident with $v$.\\
In case of (a.5) the graphs with $\delta(F) \geqslant 3$ and $e(F)=11$ are $G_{32}-G_{40}$, which are depicted in Figure 8. A simple observation shows that $\overline{C_{13}^{2}}$ contains each graph of $K_{5}^-, G_{32}-G_{40}$ and $S_{6, 2}+e$ as a subgraph but does not contain a copy of $K_{5}$. As for the graphs in (b.5), note that if the graph $D$ contains a copy of $S_{9}$, then we do not need to consider $D^{+2}$ because of $\delta(F)=2$. Thus it is easy to check that $\overline{C_{13}^{2}}$ contains each graph left in (b.5). As for the graphs in (c.5), note that if the graph $E$ contains one of graphs in $\mathcal{E}_{13}$ as a subgraph, then we do not need to consider $E^{+}$. Thus, without considering the above situations, it is not hard to show that $\overline{C_{13}^{2}}$ contains each graph left in (c.5) except those graphs in $\mathcal{L}_{13}\backslash \left\{S_{11}^{*}, S_{11} \cup K_{2}, S_{12}, S_{11}+e\right\}$. So we are done for $n=13$. Moreover, $\overline{C_{13}^{2}}$ contains each copy of $F \in \mathcal{K}(13, 9)\backslash\left\{S_{10}\right\}$ as a subgraph by Lemma \ref{lem1.6} and $\overline{C_{13}^{2}}$ contains each copy of $F \in \mathcal{K}(13, 10)$ unless $F$ contains one of graphs in $\mathcal{F}_{13}$ as a subgraph by Lemma \ref{lem1.11}.
\begin{figure}[h!]\label{fig8}
\begin{center}
\psfrag{2}{$G_{32}$}\psfrag{3}{$G_{33}$}\psfrag{4}{$G_{34}$}\psfrag{5}{$G_{35}$}\psfrag{6}{$G_{36}$}\psfrag{7}{$G_{37}$}\psfrag{8}{$G_{38}$}\psfrag{9}{$G_{39}$}
\psfrag{0}{$G_{40}$}
\includegraphics[width=160mm]{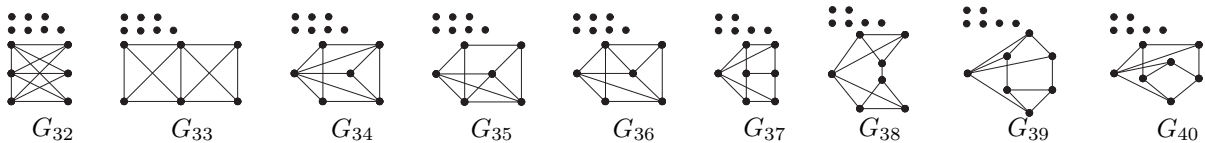} \\
  \caption{$G_{32}-G_{40}$.}
\end{center}
\end{figure}

Let $n=14$. Let $t=\left\lceil\frac{n-1}{4}\right\rceil-1=\left\lceil\frac{14-1}{4}\right\rceil-1=3$. Regardless of isolated vertices of $F$, we consider the following four cases:\\
(a.6) $\delta(F) \geqslant t+1=4$. Then the number of non-isolated vertices of $F$ is at most $\left\lfloor\frac{2(n-2)}{\left\lceil\frac{n-1}{4}\right\rceil}\right\rfloor=6$.\\
(b.6) $\delta(F)=3$. Then by Proposition \ref{prop1.5}(ii), we only need to consider $F \in\left\{S_{10}\right\}^{+3}$.\\
(c.6) $\delta(F)=2$. Then by Proposition \ref{prop1.5}(ii), we only need to consider the graphs obtained from a graph $D$ in $\mathcal{K}(13, 10)$ which contains one of graphs in $\mathcal{F}_{13}$ as a subgraph by adding an isolated vertex $v$ and two arbitrary edges incident with $v$.\\
(d.6) $\delta(F)=1$. Then by Proposition \ref{prop1.5}(i), we only need to consider the graphs obtained from a graph $E$ in $\mathcal{K}(13, 11)$ which contains one of graphs in $\mathcal{E}_{13} \cup \left\{S_{9} \cup S_{4}, S_{9} \cup K_{3}\right\}$ as a subgraph by adding an isolated vertex $v$ and an arbitrary edge incident with $v$.\\
In case of (a.6) the unique graph with $\delta(F) \geqslant 4$ and $e(F)=12$ is $\overline{K_{2}} \vee C_{4}$. A simple observation shows that $\overline{C_{14}^{2}}$ contains copies of $K_{5}^{-}$and $\overline{K_{2}} \vee C_{4}$ but does not contain a copy of $K_{5}$. As for the graphs in (b.6), there is no graph in (b.6) because of $\delta(F)=3$. As for the graphs in (c.6), note that if the graph $D$ contains a copy of $S_{10}$, then we do not need to consider $D^{+2}$ because of $\delta(F)=2$. Moreover, we do not need to consider the graph $D$ which contains a copy of $K_{5}$ belonging to $\mathcal{E}_{14}$. Thus, there is no graph in (c.6). As for the graphs in (d.6), note that if the graph $E$ contains one of graphs in $\mathcal{E}_{14}$ as a subgraph, then we do not need to consider $E^{+}$. Thus, without considering the above situations, it is not hard to show that $\overline{C_{14}^{2}}$ contains each graph left in (d.6) except those graphs in $\mathcal{L}_{14} \backslash\left\{S_{12}^{*}, S_{12} \cup K_{2}, S_{13}, S_{12}+e\right\}$. So we are done for $n=14$. Moreover, $\overline{C_{14}^{2}}$ contains each graph in $\mathcal{K}(14,10)\backslash\left\{S_{11},  K_{5}\right\}$  as a subgraph by Lemma \ref{lem1.6} and $\overline{C_{14}^{2}}$ contains each copy of $F \in \mathcal{K}(14, 11)$ unless $F$ contains one of graphs in $\mathcal{F}_{14}$ as a subgraph by Lemma \ref{lem1.11}.

Let $n=15$. Let $t=\left\lceil\frac{n-1}{4}\right\rceil-1=\left\lceil\frac{15-1}{4}\right\rceil-1=3$. Regardless of isolated vertices of $F$, we consider the following four cases:\\
(a.7) $\delta(F) \geqslant t+1=4$. Then the number of non-isolated vertices of $F$ is at most $\left\lfloor\frac{2(n-2)}{\left\lceil\frac{n-1}{4}\right\rceil}\right\rfloor=6$.\\
(b.7) $\delta(F)=3$. Then by Proposition \ref{prop1.5}(ii), we only need to consider $F \in\left\{S_{11}, K_5\right\}^{+3}$.\\
(c.7) $\delta(F)=2$. Then by Proposition \ref{prop1.5}(ii), we only need to consider the graphs obtained from a graph $D$ in $\mathcal{K}(14, 11)$ which contains one of graphs in $\mathcal{F}_{14}$ as a subgraph by adding an isolated vertex $v$ and two arbitrary edges incident with $v$.\\
(d.7) $\delta(F)=1$. Then by Proposition \ref{prop1.5}(i), we only need to consider the graphs obtained from a graph $E$ in $\mathcal{K}(14, 12)$ which contains one of graphs in $\mathcal{E}_{14} \cup \left\{S_{10} \cup S_{4}, S_{10} \cup K_{3}\right\}$ as a subgraph by adding an isolated vertex $v$ and an arbitrary edge incident with $v$.\\
In case of (a.7) the unique graph with $\delta(F) \geqslant 4$ and $e(F)=13$ is $\overline{K_{2}} \vee K_{4}^-$. A simple observation shows that $\overline{C_{15}^{2}}$ contains copies of $K_{5}$ and $\overline{K_{2}} \vee K_{4}^-$ but does not contain a copy of $K_{6}^-$. As for the graphs in (b.7), we only need to consider $F \in\left\{K_{5}\right\}^{+3}$ because of $\delta(F)=3$. As for the graphs in (c.7), note that if the graph $D$ contains a copy of $S_{11}$, then we do not need to consider $D^{+2}$ because of $\delta(F)=2$. Thus it is easy to check that $\overline{C_{15}^{2}}$ contains each graph left in (b.7) as well as (c.7). As for the graphs in (d.7), note that if the graph $E$ contains one of graphs in $\mathcal{E}_{15}$ as a subgraph, then we do not need to consider $E^{+}$. Thus, without considering the above situations, it is not hard to show that $\overline{C_{15}^{2}}$ contains each graph left in (d.7) except those graphs in $\mathcal{L}_{15} \backslash\left\{S_{13}^{*}, S_{13} \cup K_{2}, S_{14}, S_{13}+e\right\}$. So we are done for $n=15$. Moreover, $\overline{C_{15}^{2}}$ contains each graph in $\mathcal{K}(15, 11)\backslash\left\{S_{12}\right\}$  as a subgraph by Lemma \ref{lem1.6} and $\overline{C_{15}^{2}}$ contains each copy of $F \in \mathcal{K}(15, 12)$ unless $F$ contains one of graphs in $\mathcal{F}_{15}$ as a subgraph by Lemma \ref{lem1.11}.

Let $n=16$. Let $t=\left\lceil\frac{n-1}{4}\right\rceil-1=\left\lceil\frac{16-1}{4}\right\rceil-1=3$. Regardless of isolated vertices of $F$, we consider the following four cases:\\
(a.8) $\delta(F) \geqslant t+1=4$. Then the number of non-isolated vertices of $F$ is at most $\left\lfloor\frac{2(n-2)}{\left\lceil\frac{n-1}{4}\right\rceil}\right\rfloor=7$.\\
(b.8) $\delta(F)=3$. Then by Proposition \ref{prop1.5}(ii), we only need to consider $F \in\left\{S_{12}\right\}^{+3}$.\\
(c.8) $\delta(F)=2$. Then by Proposition \ref{prop1.5}(ii), we only need to consider the graphs obtained from a graph $D$ in $\mathcal{K}(15, 12)$ which contains one of graphs in $\mathcal{F}_{15}$ as a subgraph by adding an isolated vertex $v$ and two arbitrary edges incident with $v$.\\
(d.8) $\delta(F)=1$. Then by Proposition \ref{prop1.5}(i), we only need to consider the graphs obtained from a graph $E$ in $\mathcal{K}(15, 13)$ which contains one of graphs in $\mathcal{E}_{15} \cup \left\{S_{11} \cup S_{4}, S_{11} \cup K_{3}\right\}$ as a subgraph by adding an isolated vertex $v$ and an arbitrary edge incident with $v$.\\
In case of (a.8) the graphs with $\delta(F) \geqslant 4$ and $e(F)=14$ are $K_{6}^{-}, G_{14}$ and $G_{15}$. A simple observation shows that $\overline{C_{16}^{2}}$ contains copies of $K_{6}^-, G_{14}$ and $G_{15}$ but does not contain a copy of $K_{6}$. As for the graphs in (b.8), there is no graph in (b.8) because of $\delta(F)=3$. As for the graphs in (c.8), note that if the graph $D$ contains a copy of $S_{12}$, then we do not need to consider $D^{+2}$ because of $\delta(F)=2$. Thus there is no graph in (c.8). As for the graphs in (d.8), note that if the graph $E$ contains one of graphs in $\mathcal{E}_{16}$ as a subgraph, then we do not need to consider $E^{+}$. Thus, without considering the above situations, it is not hard to show that $\overline{C_{16}^{2}}$ contains each graph left in (d.8) except those graphs in $\mathcal{L}_{16} \backslash\left\{S_{14}^{*}, S_{14} \cup K_{2}, S_{15}, S_{14}+e\right\}$. So we are done for $n=16$. Moreover, $\overline{C_{16}^{2}}$ contains each graph in $\mathcal{K}(16,12)\backslash\left\{S_{13}\right\}$  as a subgraph by Lemma \ref{lem1.6} and $\overline{C_{16}^{2}}$ contains each copy of $F \in \mathcal{K}(16, 13)$ unless $F$ contains $S_{13}$ as a subgraph by Lemma \ref{lem1.12}.

Let $n=17$. Let $t=\left\lceil\frac{n-1}{4}\right\rceil-1=\left\lceil\frac{17-1}{4}\right\rceil-1=3$. Regardless of isolated vertices of $F$, we consider the following four cases:\\
(a.9) $\delta(F) \geqslant t+1=4$. Then the number of non-isolated vertices of $F$ is at most $\left\lfloor\frac{2(n-2)}{\left\lceil\frac{n-1}{4}\right\rceil}\right\rfloor=7$.\\
(b.9) $\delta(F)=3$. Then by Proposition \ref{prop1.5}(ii), we only need to consider $F \in\left\{S_{13}\right\}^{+3}$.\\
(c.9) $\delta(F)=2$. Then by Proposition \ref{prop1.5}(ii), we only need to consider the graphs obtained from a graph $D$ in $\mathcal{K}(16, 13)$ which contains $S_{13}$ as a subgraph by adding an isolated vertex $v$ and two arbitrary edges incident with $v$.\\
(d.9) $\delta(F)=1$. Then by Proposition \ref{prop1.5}(i), we only need to consider the graphs obtained from a graph $E$ in $\mathcal{K}(16, 14)$ which contains one of graphs in $\mathcal{E}_{16} \cup \left\{S_{12} \cup S_{4}, S_{12} \cup K_{3}\right\}$ as a subgraph by adding an isolated vertex $v$ and an arbitrary edge incident with $v$.\\
In case of (a.9) the graphs with $\delta(F) \geqslant 4$ and $e(F)=15$ are $G_{10} \vee \overline{K_1}, \overline{K_3} \vee P_{4}, \overline{K_2} \vee C_{5}, K_{6}, G_{41}, G_{42}$ and $G_{43}$, where $G_{41}-G_{43}$ are depicted in Figure 9. A simple observation shows that $\overline{C_{17}^{2}}$ contains copies of $G_{10} \vee \overline{K_1}, G_{41}, G_{42}, \overline{K_3} \vee P_{4}, \overline{K_2} \vee C_{5}, K_{6}^-$ and $G_{43}$ but does not contain a copy of $K_{6}$. As for the graphs in (b.9), there is no graph in (b.9) because of $\delta(F)=3$. As for the graphs in (c.9), note that if the graph $D$ contains a copy of $S_{13}$, then we do not need to consider $D^{+2}$ because of $\delta(F)=2$. Thus there is no graph in (c.9). As for the graphs in (d.9), note that if the graph $E$ contains one of graphs in $\mathcal{E}_{17}$ as a subgraph, then we do not need to consider $E^{+}$. Thus, without considering the above situations, it is not hard to show that $\overline{C_{17}^{2}}$ contains each graph left in (d.9) except those graphs in $\mathcal{L}_{17} \backslash\left\{S_{15}^{*}, S_{15} \cup K_{2}, S_{16}, S_{15}+e\right\}$. So we are done for $n=17$. Moreover, $\overline{C_{17}^{2}}$ contains each graph in $\mathcal{K}(17,13)\backslash\left\{S_{14}\right\}$  as a subgraph by Lemma \ref{lem1.6} and $\overline{C_{17}^{2}}$ contains each copy of $F \in \mathcal{K}(17, 14)$ unless $F$ contains $S_{14}$ as a subgraph by Lemma \ref{lem1.12}. Obviously, $\overline{C_{17}^{2}}$ contains each graph in $\mathcal{K}(17, 12)$ as a subgraph.
\begin{figure}[h!]\label{fig9}
\begin{center}
\psfrag{1}{$G_{41}$}\psfrag{2}{$G_{42}$}\psfrag{3}{$G_{43}$}
\includegraphics[width=90mm]{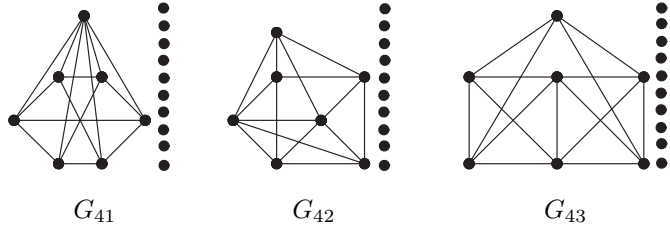} \\
  \caption{Graphs $G_{41}, G_{42}$ and $G_{43}$.}
\end{center}
\end{figure}

Let $n=18$. Let $t=\left\lceil\frac{n-1}{4}\right\rceil-1=\left\lceil\frac{18-1}{4}\right\rceil-1=4$. Regardless of isolated vertices of $F$, we consider the following five cases:\\
(a.10) $\delta(F) \geqslant t+1=5$. Then the number of non-isolated vertices of $F$ is at most $\left\lfloor\frac{2(n-2)}{\left\lceil\frac{n-1}{4}\right\rceil}\right\rfloor=6$.\\
(b.10) $\delta(F)=4$. Then by Proposition \ref{prop1.5}(ii), $\overline{C_{18}^{2}}$ contains each copy of $F \in \mathcal{K}(17, 12)^{+4}$.\\
(c.10) $\delta(F)=3$. Then by Proposition \ref{prop1.5}(ii), we only need to consider $F \in\left\{S_{14}\right\}^{+3}$.\\
(d.10) $\delta(F)=2$. Then by Proposition \ref{prop1.5}(ii), we only need to consider the graphs obtained from a graph $D$ in $\mathcal{K}(17, 14)$ which contains $S_{14}$ as a subgraph by adding an isolated vertex $v$ and two arbitrary edges incident with $v$.\\
(e.10) $\delta(F)=1$. Then by Proposition \ref{prop1.5}(i), we only need to consider the graphs obtained from a graph $E$ in $\mathcal{K}(17, 15)$ which contains one of graphs in $\mathcal{E}_{17} \cup \left\{S_{13} \cup S_{4}, S_{13} \cup K_{3}\right\}$ as a subgraph by adding an isolated vertex $v$ and an arbitrary edge incident with $v$.\\
In case of (a.10) there does not exist a graph with $\delta(F) \geqslant 5$ and $e(F)=16$. A simple observation shows that $\overline{C_{18}^{2}}$ contains copies of $K_{6} \cup K_{2}$ and $K_{6}^+$. As for the graphs in (c.10), there is no graph in (c.10) because of $\delta(F)=3$. As for the graphs in (d.10), note that if the graph $D$ contains a copy of $S_{14}$, then we do not need to consider $D^{+2}$ because of $\delta(F)=2$. Thus there is no graph in (d.10). As for the graphs in (e.10), note that if the graph $E$ contains one of graphs in $\mathcal{E}_{18}$ as a subgraph, then we do not need to consider $E^{+}$. Thus, without considering the above situations, it is not hard to show that $\overline{C_{18}^{2}}$ contains each graph left in (e.10) except those graphs in $\mathcal{L}_{18} \backslash\left\{S_{16}^{*}, S_{16} \cup K_{2}, S_{17}, S_{16}+e\right\}$. So we are done for $n=18$.
\end{proof}

\begin{lem}\label{lem1.14}
Let $n \geqslant 18$. Then $\overline{C_{n}^{2}}$ contains each copy of $F \in \mathcal{K}(n, n-2)\backslash \left\{S_{n-4} \cup S_4, S_{n-4} \cup K_3\right\}$ unless $F$ contains $S_{n-3}$ as a subgraph.
\end{lem}

\begin{proof}
We prove the lemma by induction on $n$. For $n=18$, the lemma follows from Lemma \ref{lem1.13}. Suppose it is true for $n-1$. Then by induction hypothesis, $\overline{C_{n-1}^{2}}$ contains each copy of $F \in \mathcal{K}(n-1, n-3)\backslash \left\{S_{n-5} \cup S_4, S_{n-5} \cup K_3\right\}$ unless $F$ contains $S_{n-4}$ as a subgraph. Moreover, $\overline{C_{n-1}^{2}}$ contains each graph in $\mathcal{K}(n-1, n-5) \backslash\left\{S_{n-4}\right\}$ and contains each graph in $\mathcal{K}(n-1, n-6)$ by Lemma \ref{lem1.6}. And $\overline{C_{n-1}^{2}}$ contains each copy of $F \in \mathcal{K}(n-1, n-4)$ unless $F$ contains $S_{n-4}$ as a subgraph by Lemma \ref{lem1.12}.

It is now sufficient to show that $\overline{C_{n}^{2}}$ contains each copy of $F \in \mathcal{K}(n, n-2)\backslash \left\{S_{n-4} \cup S_4, S_{n-4} \cup K_3\right\}$ unless $F$ contains $S_{n-3}$ as a subgraph. Let $t=\left\lceil\frac{n-1}{4}\right\rceil-1$. Regardless of isolated vertices of $F$, we consider the following two cases:\\
(a) $1 \leq \delta(F) \leq t$. Applying Proposition \ref{prop1.5}, it is easy to see that $\overline{C_{n}^{2}}$ contains each copy of $F$ unless $F \in \mathcal{L}_{n} \backslash\left\{S_{n-2} \cup K_{2}, S_{n-2}^{*}, S_{n-2}+e, S_{n-1}\right\}$, in which each graph $F$ contains either $S_{n-3}$ as a subgraph or $F \in\left\{S_{n-4} \cup S_4, S_{n-4} \cup K_3\right\}$.\\
(b) $\delta(F) \geqslant t+1$. Then the number of non-isolated vertices of $F$ is at most $\left\lfloor\frac{2(n-2)}{\left\lceil\frac{n-1}{4}\right\rceil}\right\rfloor$.\\
We know that $\overline{C_{n}^{2}}$ contains a copy of $K_{\lfloor\frac{n}{3}\rfloor}$ by Fact \ref{fact1.7}. Now, if $n \geqslant 22$, then we have $\left\lfloor\frac{2(n-2)}{\left\lceil\frac{n-1}{4}\right\rceil}\right\rfloor \leqslant\left\lfloor\frac{n}{3}\right\rfloor$. Thus $\overline{C_{n}^{2}}$ contains a copy of $F$ for $n \geqslant 22$. Let $n=19$. Then the number of non-isolated vertices of $F$ is at most $\left\lfloor\frac{2(n-2)}{\left\lceil\frac{n-1}{4}\right\rceil}\right\rfloor=6$ and $e(F)=17$. Clearly, such $F$ does not exist. Let $n=20$. Then the number of non-isolated vertices of $F$ is at most $\left\lfloor\frac{2(n-2)}{\left\lceil\frac{n-1}{4}\right\rceil}\right\rfloor=7$ and $e(F)=18$. Actually $F$ must be the graph $\overline{K_{2}} \vee (K_1 \vee C_4)$ because of $\delta(F)\geqslant\left\lceil\frac{n-1}{4}\right\rceil=5$ and one can easily check that $\overline{C_{20}^{2}}$ contains a copy of $F$. Let $n=21$. Then the number of non-isolated vertices of $F$ is at most $\left\lfloor\frac{2(n-2)}{\left\lceil\frac{n-1}{4}\right\rceil}\right\rfloor=7$ and $e(F)=19$. By Fact \ref{fact1.7}, $\overline{C_{21}^{2}}$ contains a copy of $K_{7}$ which has 21 edges and one can easily check that $\overline{C_{21}^{2}}$ contains a copy of $F$.

This completes the proof.
\end{proof}

\section{\normalsize Proof for Theorem \ref{thm1.1} and Theorem \ref{thm1.2}}

In this {section}, we will prove the main theorems based on the results in the last {section}.

\noindent{\bf Proof of Theorem \ref{thm1.1}:}
If $e(G) \geqslant C_{n-1}^{2}+3=\frac{n^{2}-3 n+8}{2}$, then we have that $G$ contains the $2$-power of a Hamilton cycle or $G$ is the graph $K_{n} \backslash E\left(S_{n-3}\right)$ by Theorem \ref{thm1.4}. Hence in this case the proof is done.\\
If $e(G)=\frac{n^{2}-3n+6}{2}=C_{n-1}^{2}+2$ and assume that $G$ does not contain the $2$-power of a Hamilton cycle, by Lemma \ref{lem1.12}, we have that the graph $K_{n} \backslash E\left(S_{n-3}\right)$ contains $G$ as a subgraph. Hence in this case the proof is done.\\
If $e(G)=\frac{n^{2}-3n+4}{2}=C_{n-1}^{2}+1$ and assume that $G$ does not contain the $2$-power of a Hamilton cycle, by Lemma \ref{lem1.14}, we have that $Y_n$ contains $G$ as a subgraph.

This completes the proof.
\qed

\noindent{\bf Proof of Theorem \ref{thm1.2}:}
Note that $$\mu(G) \geqslant \mu\left(K_{n} \backslash E\left(S_{n-3}\right)\right)>\mu\left(K_{n-1}\right)=n-2$$ by the Perron-Frobenius theorem. From Theorem \ref{thm1.9}, we have that $G$ contains a Hamilton cycle or $G=K_{n-1}+e$. Since $\mu\left(K_{n-1}+e\right)<\mu\left(K_{n} \backslash E\left(S_{n-3}\right)\right)$, we have $G$ contains a Hamilton cycle and obviously $G$ is connected.

Let $m=e(G)$. From Lemma \ref{lem1.8}, we have $\mu(G) \leqslant \sqrt{2m-n+1}$ with equality if and only if $G=S_{n}$ or $K_{n}$. Note that $\mu\left(S_{n}\right)=\sqrt{n-1}<n-2$, then if the equality holds, then $G=K_{n}$ and in this case the proof is done. If the strict inequality holds, then $2m-n+1>(n-2)^{2}$ and so $m>\frac{n^{2}-3 n+3}{2}$. By Theorem \ref{thm1.1}, we have that G contains the $2$-power of a Hamilton cycle or $Y_{n}$ contains $G$ as a subgraph, where $Y_{n}$ is a family of $K_{n} \backslash E\left(S_{n-3}\right), K_{n} \backslash E\left(S_{n-4} \cup S_4\right)$ and $K_{n} \backslash E\left(S_{n-4} \cup K_3\right)$. One can easily check that $$\mu\left(K_{n} \backslash E\left(S_{n-4}\cup S_4\right)\right)<\mu\left(K_{n} \backslash E\left(S_{n-3}\right)\right)$$ and $$\mu\left(K_{n} \backslash E\left(S_{n-4} \cup K_3\right)\right)<\mu\left(K_{n} \backslash E\left(S_{n-3}\right)\right).$$ Combining the Perron-Frobenius theorem and $\mu(G) \geqslant \mu\left(K_{n} \backslash E\left(S_{n-3}\right)\right)$, we have  that $G$ contains the $2$-power of a Hamilton cycle or $G=K_{n} \backslash E\left(S_{n-3}\right)$. This completes the proof.
\qed

\section{\normalsize Proof for Theorem \ref{thm1.3}}

In this section, we consider the relationship between the complement of a graph $G$ and its spectral radius $\mu(G)$ based on the concept of $k$-closure $\mathcal{C}_{k}(G)$, which is formally introduced by Bondy and Chvatal in \cite{5}.

\noindent{\bf Proof of Theorem \ref{thm1.3}:}
We prove by contradiction. For short, let $H=\mathcal{C}_{n}(G)$. Assume that $G$ does not contain the $2$-power of a Hamilton cycle $C_{n}^{2}$. Note that $\mu(\overline{G}) \leqslant \sqrt{n-5}<\sqrt{n-2}$ and $\overline{G}$ contains a copy of $\overline{H}$. By the Perron-Frobenius theorem, we obtain $\mu(\overline{H}) \leqslant \mu(\overline{G})<\sqrt{n-2}$. From Theorem \ref{thm1.9}, we have that $G$ contains a Hamilton cycle and obviously $G$ is connected because $G \neq K_{n-1}+e$ in this case $$\mu\left(\overline{K_{n+1}+e}\right)=\mu\left(S_{n-1}\right)=\sqrt{n-2}>\sqrt{n-5}.$$ Now we consider the following two cases:\\
(i) If $\overline{C_{n}^{2}}$ does not contain a copy of $\overline{H}$, we have that $e(\overline{H}) \geqslant n-4$ by Theorem \ref{thm1.4}. Now the main property of $\mathcal{C}_{n}(G)=H$ gives $d_{H}(u)+d_{H}(v) \leqslant n-1$ for every pair of nonadjacent vertices $u$ and $v$ of $H$; thus, $$d_{\overline{H}}(u)+d_{\overline{H}}(v)=n-1-d_{H}(u)+n-1-d_{H}(v) \geqslant n-1$$ for every edge $uv \in E(\overline{H})$. Summing these inequalities for all edges $uv \in E(\overline{H})$, we obtain $$\sum\limits_{uv \in E(\overline{H})}\left(d_{\overline{H}}(u)+d_{\overline{H}}(v)\right) \geqslant (n-1)e(\overline{H})$$ and since each term $d_{\overline{H}}(u)$ appears in the left-hand sum precisely $d_{\overline{H}}(u)$ times, we see that $$\sum\limits_{v \in V(\overline{H})} d_{\overline{H}}^{2}(v)=\sum\limits_{uv \in E(\overline{H})} \left(d_{\overline{H}}(u)+d_{\overline{H}}(v)\right) \geqslant (n-1)e(\overline{H}).$$ By Lemma \ref{lem1.10}, we have $$n\mu^{2}(\overline{H}) \geqslant \sum\limits_{v \in V(\overline{H})} d_{\overline{H}}^{2}(v) \geqslant (n-1)e(\overline{H}).$$ Since $\overline{H} \subset \overline{G}$, we have $\mu(\overline{H}) \leqslant \mu(\overline{G}) \leqslant \sqrt{n-5}$ and so, $$n(n-5) \geqslant n \mu^{2}(\overline{G}) \geqslant n \mu^{2}(\overline{H}) \geqslant (n-1)e(\overline{H}).$$ This easily gives $$e(\overline{H}) \leqslant \frac{n(n-5)}{n-1}<n-4, n \geqslant 18,$$ a contradiction.\\
(ii) From (i), we must have that $\overline{C_{n}^{2}}$ contains a copy of $\overline{H}$ but $\overline{C_{n}^{2}}$ does not contain a copy of $\overline{G}$. In this case our aim is to prove $H=K_{n}$. Assume that there exists an edge $uv \in E(\overline{G})$ satisfying $d_{\overline{G}}(u)+d_{\overline{G}}(v) \geqslant n-1$. Now we consider $\overline{G}$ with the smallest spectral radius satisfying that $\overline{C_{n}^{2}}$ does not contain a copy of $\overline{G}$. Firstly, we consider three graphs $H_{1}, H_{2}$ and $H_{3}$ with $n$ vertices and $n-2$ edges, which are depicted in Figure 10. Let $g(\lambda), h(\lambda)$ and $f(\lambda)$ be the characteristic polynomial of $H_{1}, H_{2}$ and $H_{3}$ respectively. One can easily know that $$g(\lambda)=\lambda^{5}-(b+2c+a+1) \lambda^{3}-2c \lambda^{2}+(ab+ac+bc) \lambda$$ and $$h(\lambda)=\lambda^{5}-(b+2c+a+1) \lambda^{3}-(2c-2) \lambda^{2}+(ab+ac+bc+2c-1) \lambda,$$ thus $g(\lambda)-h(\lambda)=-\lambda(2c-1+2 \lambda)$. Note that $\mu(G)$ is the largest root of the characteristic polynomial of $G$. Thus it is obvious that $\mu\left(H_{2}\right)<\mu\left(H_{1}\right)$. Then it is clearly that $\overline{G}$ with the smallest spectral radius satisfying that $\overline{C_{n}^{2}}$ does not contain a copy of $\overline{G}$ is exactly $H_{3}$. Similarly, we have $f(\lambda)=\lambda^{4}+(-b-1-a) \lambda^{2}+ab$ with $a+b+2=n-1$ and at this time $$\mu(\overline{G})=\mu\left(H_{3}\right)=\frac{\sqrt{2+2(n-3)+2 \sqrt{2(n-3)+1+(a-b)^{2}}}}{2}.$$ To simplify the proof, we denote $H_{3}$ by $T_{a, b}$ relying on $a$ and $b$ satisfying $a+b+2=n-1$. By Lemma \ref{lem1.14}, we have that $\overline{G}=T_{n-5, 2}$ and $$\mu\left(T_{n-5, 2}\right)=\frac{\sqrt{2n-4+2 \sqrt{n^{2}-12n+44}}}{2}>\sqrt{n-4}>\sqrt{n-5},$$ a contradiction. Consequently, we have $d_{G}(u)+d_{G}(v) \geqslant n$ for every pair of nonadjacent vertices $u$ and $v$ of $G$. Then $H=\mathcal{C}_{n}(G)=K_{n}$.

This completes the proof.
\qed
\begin{figure}[h!]\label{fig10}
\begin{center}
\psfrag{0}{$\cdots$}\psfrag{1}{$H_{1}$}\psfrag{2}{$H_{2}$}\psfrag{3}{$H_{3}$}\psfrag{u}{$u$}\psfrag{v}{$v$}\psfrag{a}{$a$}\psfrag{b}{$b$}\psfrag{c}{$c$}
\psfrag{d}{$n-2-a-b-c$}\psfrag{e}{$a+1$}\psfrag{f}{$c-1$}\psfrag{g}{$b+1$}\psfrag{h}{$n-3-a-b-c$}
\includegraphics[width=100mm]{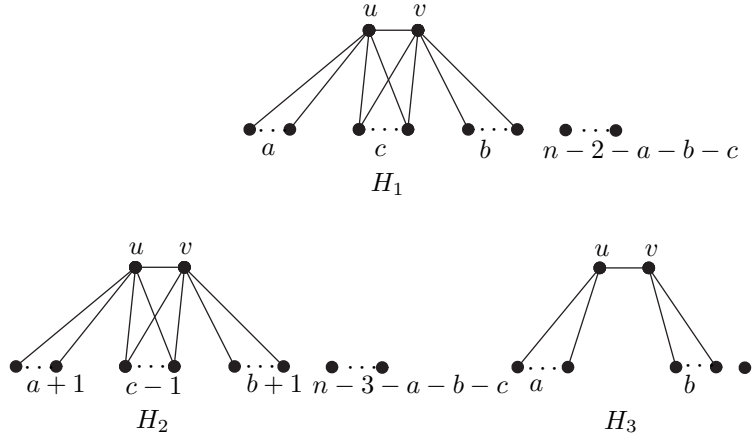} \\
  \caption{Graphs $H_{1}, H_{2}$ and $H_{3}$.}
\end{center}
\end{figure}

In this paper, we consider the spectral conditions for a graph containing $C_n^2$ as a subgraph. What can we say for the problems for  $C_n^k$? Furthermore, what happens if the minimum degree is fixed? We leave them for further research.
We will also consider maximum signless laplacian ($p$-Laplacian) spectral radius version of the problem among the same family of graphs in the near future.


\section*{Acknowledgments}
This work is supported by NSFC (Nos. 11871479, 12001544, 12071484) and Natural Science Foundation of Hunan Province (Nos. 2020JJ4675,  2021JJ40707).

{
}


\begin{thebibliography}{99}
\small \setlength{\itemsep}{-.8mm}
\bibitem{bondy} J.A. Bondy, U.S.R. Murty, Graph Theory, in: Graduate Texts in Mathematics, Vol. 244, Springer, 2008.
\bibitem{12} J.A. Bondy, Properties of graphs with constraints on degrees, Studia Sci. Math. Hunger. 4 (1969) 473-475.

\bibitem{5} A. Bondy, V. Chvatal, A method in graph theory, Discrete Math. 15 (1976) 111-135.
\bibitem{8} D. de Caen, An upper bound on the sum of squares of degrees in a graph, Discrete Math. 185 (1998) 245-248.

\bibitem{13} V. Chvatal, On Hamilton's ideals, J. Combin. Theory Ser. B 12 (1972) 163-168.
\bibitem{9} D. Cvetkovi\'{c}, M. Doob, H. Sachs, Spectra of Graphs--Theory and Application,  Heidelberg, 1995.
\bibitem{10} D. Cvetkovi\'{c}, P. Rowlinson, S. Simi\'{c}, Signless Laplacians of finite graphs, Linear Algebra Appl. 423 (2007) 155-171.



\bibitem{dirac} G. A. Dirac, Some theorems on abstract graphs, Proc. London Math. Soc. 2 (1952)  68-81.
\bibitem{erdos} P. Erd{\"o}s, Problem 9, in ``Theory of Graphs and Its Applications'' (M. Fieldler, Ed.),p. 159, Czech. Acad. Sci. Publ., Prague, 1964.

\bibitem{f1} G.H. Fan, R. H{\"a}ggkvist, The square of a hamiltonian cycle, SIAM J. Discrete Math. 7 (1994) 203-212.
\bibitem{f2} G.H. Fan, H.A. Kierstead, The square of paths and cycles, J. Combinatorial Theory Ser. B 63 (1995) 55-64.
\bibitem{f3} G.H. Fan, H.A. Kierstead, Hamiltonian Square-Paths, J. Combinatorial Theory Ser. B 67 (1996) 167-182.

\bibitem{fau} R.J. Faudree, R.J. Gould, M.S. Jacobson, R.H. Schelp, On a problem of Paul Seymour, in: V.R. Kulli (Ed.), Recent Advances in Graph Theory, Vishwa
International Publication, 1991, pp. 197-215.
\bibitem{19} L.H. Feng, P.L. Zhang, H. Liu, W.J. Liu, M.M. Liu, Y.Q. Hu, Spectral conditions for some graphical properties, Linear Algebra Appl. 524 (2017) 182-198.
\bibitem{4} M. Fiedler, V. Nikiforov, Spectral radius and Hamiltonicity of graphs, Linear Algebra Appl. 432 (2010) 2170-2173.





\bibitem{gao} J. Gao, X. Hou, The spectral radius of graphs without long cycles, Linear Algebra Appl. 566 (2019) 17-33.
\bibitem{11} J. van den Heuvel, Hamliton cycles and eigenvalues of graphs, Linear Algebra Appl. 226-228 (1995) 723-730.
\bibitem{2} M. Hofmeister, Spectral radius and degree sequence, Math. Nachr. 139 (1988) 37-44.
\bibitem{1} Y. Hong, A bound on the spectral radius of graphs, Linear Algebra Appl. 108 (1988) 135-139.
\bibitem{3} Z.U. Khan, L.T. Yuan, A note on the $2$-power of Hamilton cycles, submitted to Discrete Math.
\bibitem{kie} H.A. Kierstead, J. Quintana, Square Hamiltonian cycles in graphs with maximal 4-cliques, Discrete Math. 178 (1998) 81-92.
\bibitem{k1} J. Koml\'{o}s, G.N. S\'{a}rk{\"o}zy, E. Szemer\'{e}di, On the Posa-Seymour conjecture, Journal of Graph Theory 29 (1998) 167-176.
\bibitem{k2} J. Koml\'{o}s, G.N. S\'{a}rk{\"o}zy, E. Szemer\'{e}di, On the square of a Hamiltonian cycle in dense graphs, Random Structures and Algorithms 9 (1996) 193-211.
\bibitem{k3} J. Koml\'{o}s, G.N. S\'{a}rk{\"o}zy, E. Szemer\'{e}di, Proof of the Seymour conjecture for large graphs, Annals of Combinatorics 2 (1998) 43-60.


\bibitem{L} I. Levitt, G.N. S\'{a}rk{\"o}zy, E. Szemer\'{e}di, How to avoid using the Regularity Lemma: P\'{o}sa's conjecture revisited, Discrete Math. 310 (2010) 630-641.

   \bibitem{Libinlong1}  B.L.   Li, B. Ning, Spectral analogues of Erd\"os' and Moon-Moser's theorems on Hamilton cycles,  Linear Multilinear Algebra 64(11) (2016) 2252--2269.

 \bibitem{Libinlong2} B.L.   Li, B. Ning, X. Peng,  Extremal problems on the Hamiltonicity of claw-free graphs, Discrete Math. 341(10) (2018)  2774--2788.


 \bibitem{Linhuiqiu1}   H.Q. Lin,   H.T. Guo,   A spectral condition for odd cycles in non-bipartite graphs,  Linear Algebra Appl. 631 (2021)  83--93.

 \bibitem{Linhuiqiu2}  H.Q. Lin,  B. Ning, B. Wu,  Eigenvalues and triangles in graphs,  Combin. Probab. Comput. 30(2) (2021) 258--270.

\bibitem{15} B. Mohar, A domain monotonicity theorem for graphs and hamiltonicity, Discrete Appl. Math. 36 (1992) 169-177.
\bibitem{6} V. Nikiforov, The sum of the squares of degrees: sharp asymptotics, Discrete Math. 307 (2007) 3187-3193.
\bibitem{ni} V. Nikiforov, The spectral radius of graphs without paths and cycles of specified length, Linear Algebra Appl. 432 (2010) 2243-2256.
\bibitem{7} O. Ore, Note on Hamilton circuits, Amer. Math. Monthly 67 (1960) 55.
\bibitem{14} O. Ore, Arc coverings of graphs. Ann. Mat. Pura Appl. (4) 55 (1961) 315-321.
\bibitem{16} R. Stanley, A bound on the spectral radius of graphs with $e$ edges, Linear Algebra Appl. 87 (1987) 267-269.
\bibitem{sey} P. Seymour, Problem section, in: T.P. McDonough, V.C. Mavron (Eds.), Combinatorics: Proceedings of the British Combinatorial Conference 1973,
Cambridge University Press, 1974, pp. 201-202.
\bibitem{18} J.F. Wang, J. Wang, X.G. Liu, F. Belardo, Graphs whose $A_{\alpha}$-spectral radius does not exceed $2$, Discuss. Math., Graph Theory 40 (2020) 677-690.

\bibitem{20} M.Q. Zhai, B. Wang, L.F. Fang, The spectral Tur\'{a}n problem about graphs with no $6$-cycle, Linear Algebra Appl. 590 (2020) 22-31.

\bibitem{17} Y.J. Zhu, F. Tian, A Generalization of the Bondy-Chvatal Theorem on the $k$-Closure, J. Combin. Theory Ser. B 35 (1983) 247-255.
\end{thebibliography}
\end{document}